\documentclass[11pt]{amsart}
\usepackage[margin=1in]{geometry}

\usepackage{amssymb}
\usepackage{amsthm}
\usepackage{amsmath}
\usepackage{mathrsfs}
\usepackage{amsbsy}
\usepackage[all]{xy}
\usepackage{bm}
\usepackage{hyperref}
\usepackage{tikz}
\usepackage{array}
\usepackage{float}
\usepackage{enumerate}
\usepackage{xcolor}
\usepackage{bbm}

\hypersetup{colorlinks=true, citecolor=darkblue, linkcolor=darkblue,urlcolor=darkblue}

\definecolor{darkblue}{rgb}{0.5,0,1.0}

\usepackage{hhline}
\allowdisplaybreaks
\usepackage[noadjust]{cite}
\usepackage{asymptote}

\usepackage{caption}
\usepackage{tabu}
\usepackage{diagbox}
\usepackage[noabbrev,capitalise,nameinlink]{cleveref}

\newcommand{\ZZ}{\mathbb{Z}}

\newenvironment{enumerate*}
  {\begin{enumerate}[(i)]
    \setlength{\itemsep}{10pt}
    \setlength{\parskip}{0pt}}
  {\end{enumerate}}

\newtheorem{theorem}{Theorem}[section]
\newtheorem{proposition}[theorem]{Proposition}

\newtheorem{lemma}[theorem]{Lemma}

\theoremstyle{definition}
\newtheorem{definition}[theorem]{Definition}
\newtheorem{remark}[theorem]{Remark}
\newtheorem{example}[theorem]{Example}

\DeclareMathOperator{\TPro}{TPro}
\DeclareMathOperator{\Pro}{Pro}
\DeclareMathOperator{\Bro}{Bro}
\DeclareMathOperator{\lcm}{lcm}
\DeclareMathOperator{\jdt}{jdt}
\DeclareMathOperator{\Rot}{Rot}
\DeclareMathOperator{\cyc}{cyc}
\DeclareMathOperator{\rev}{rev}
\newcommand{\Coll}{\mathrm{Coll}}
\newcommand{\Orb}{\mathrm{Orb}}
\newcommand{\Comp}{\mathsf{Comp}}
\newcommand{\Path}{\mathsf{Path}}
\newcommand{\Cycle}{\mathsf{Cycle}}
\newcommand{\s}{\mathsf{s}}
\newcommand{\cc}{\mathsf{c}}
\newcommand{\vv}{{{\bf v}}}
\newcommand{\stand}{\mathrm{stand}}

\newcommand{\dfn}[1]{\textcolor{blue}{\emph{#1}}}

\begin{document}

\title[]{Permutoric Promotion: \\ Gliding Globs, Sliding Stones, and Colliding Coins}
\subjclass[2010]{}

\author[]{Colin Defant}
\address[]{Department of Mathematics, Massachusetts Institute of Technology, Cambridge, MA 02139, USA}
\email{colindefant@gmail.com}

\author[]{Rachana Madhukara}
\address[]{Department of Mathematics, Massachusetts Institute of Technology, Cambridge, MA 02139, USA}
\email{rachanam@mit.edu}

\author[]{Hugh Thomas}
\address[]{Lacim, UQAM, Montréal, QC, H3C 3P8 Canada}
\email{thomas.hugh\_r@uqam.ca}

\maketitle

\begin{abstract}
The first author recently introduced \emph{toric promotion}, an operator that acts on the labelings of a graph $G$ and serves as a cyclic analogue of Sch\"utzenberger's promotion operator. Toric promotion is defined as the composition of certain \emph{toggle} operators, listed in a natural cyclic order. We consider more general \emph{permutoric promotion} operators, which are defined as compositions of the same toggles, but in permuted orders. We settle a conjecture of the first author by determining the orders of all permutoric promotion operators when $G$ is a path graph. In fact, we completely characterize the orbit structures of these operators, showing that they satisfy the cyclic sieving phenomenon. The first half of our proof requires us to introduce and analyze new \emph{broken promotion} operators, which can be interpreted via globs of liquid gliding on a path graph. For the latter half of our proof, we reformulate the dynamics of permutoric promotion via stones sliding along a cycle graph and coins colliding with each other on a path graph.  
\end{abstract}

\section{Introduction}\label{SecIntro}

In his study of the Robinson--Schensted--Knuth correspondence, Sch\"utzenberger \cite{Schutzenberger1, Schutzenberger2, Schutzenberger3} introduced a beautiful bijective operator called \emph{promotion}, which acts on the set of linear extensions of a finite poset. Haiman \cite{Haiman} and Malvenuto--Reutenauer \cite{Malvenuto} found that promotion could be defined as a composition of local \emph{toggle operators} (also called \emph{Bender--Knuth involutions}). There are now several articles connecting promotion to other areas \cite{Ayyer, Edelman, HopkinsRubey, Huang, Petersen, Poznanovic, Rhoades, StanleyPromotion, StrikerWilliams} and generalizing promotion in different directions \cite{Ayyer, Bernstein, DefantPromotionSorting, Dilks, Dilks2, StanleyPromotion}. Promotion is now one of the most extensively studied operators in the field of dynamical algebraic combinatorics.

Following the approach first considered by Malvenuto and Reutenauer \cite{Malvenuto}, we define promotion on labelings of graphs instead of linear extensions of posets. All graphs in this article are assumed to be simple. Let $G=(V,E)$ be a graph with $n$ vertices. A \dfn{labeling} of $G$ is a bijection $V\to \mathbb Z/n\mathbb Z$. We denote the set of labelings of $G$ by $\Lambda_G$. Given distinct $a,b\in\mathbb Z/n\mathbb Z$, let $(a\,\,b)$ be the transposition that swaps $a$ and $b$. For $i\in\ZZ/n\ZZ$, the \dfn{toggle} operator $\tau_i\colon \Lambda_G\to\Lambda_G$ is defined by 
\[
\tau_i(\sigma)=\begin{cases} (i\,\,i+1)\circ \sigma & \mbox{if } \{\sigma^{-1}(i),\sigma^{-1}(i+1)\}\not\in E; \\ \sigma & \mbox{if } \{\sigma^{-1}(i),\sigma^{-1}(i+1)\}\in E. \end{cases}
\]
In other words, $\tau_i$ swaps the labels $i$ and $i+1$ if those labels are assigned to nonadjacent vertices of $G$, and it does nothing otherwise. Define \dfn{promotion} to be the operator $\Pro\colon\Lambda_G\to\Lambda_G$ given by \[\Pro=\tau_{n-1}\cdots\tau_2\tau_1.\] Here and in the sequel, concatenation of operators represents composition. 

A recent trend in algebraic combinatorics aims to find cyclic analogues of more traditional ``linear'' objects (see \cite{AdinCyclic,Develin} and the references therein). The first author recently defined a cyclic analogue of promotion called \dfn{toric promotion} \cite{DefantToric}; this is the operator $\TPro\colon\Lambda_G\to\Lambda_G$ given by \[\TPro=\tau_{n}\tau_{n-1}\cdots\tau_2\tau_1=\tau_n\Pro.\] 

The first author proved the following theorem, which reveals that toric promotion has remarkably nice dynamical properties when $G$ is a forest. 

\begin{theorem}[\cite{DefantToric}]\label{thm:toric_main}
Let $G$ be a forest with $n\geq 2$ vertices, and let $\sigma\in\Lambda_G$ be a labeling. The orbit of toric promotion containing $\sigma$ has size \[(n-1)\frac{t}{\gcd(t,n)},\] where $t$ is the number of vertices in the connected component of $G$ containing $\sigma^{-1}(1)$. In particular, if $G$ is a tree, then every orbit of $\TPro\colon\Lambda_G\to\Lambda_G$ has size $n-1$.
\end{theorem}

\Cref{thm:toric_main} stands in stark contrast to the wild dynamics of promotion on most forests. For example, even when $G$ is a path graph with $7$ vertices, the order of $\Pro\colon\Lambda_G\to\Lambda_G$ is $3224590642072800$, whereas all orbits of $\TPro\colon\Lambda_G\to\Lambda_G$ have size $6$. 

In \cite{DefantToric}, the first author (taking a suggestion from Tom Roby) proposed studying a generalization of toric promotion in which the toggle operators $\tau_1,\ldots,\tau_n$ can be composed in any order. In what follows, we let $[n]=\{1,\ldots,n\}$. 

\begin{definition}
Let $G$ be a graph with $n$ vertices, and let $\pi\colon[n]\to \mathbb Z/n\mathbb Z$ be a bijection. The \dfn{permutoric promotion} operator $\TPro_\pi\colon\Lambda_G\to\Lambda_G$ is defined by \[\TPro_\pi=\tau_{\pi(n)}\cdots\tau_{\pi(2)}\tau_{\pi(1)}.\] 
\end{definition}

One would ideally hope to have an extension of \Cref{thm:toric_main} to arbitrary permutoric promotion operators. Unfortunately, trying to completely describe the orbit structure of $\TPro_\pi\colon\Lambda_G\to\Lambda_G$ for arbitrary forests $G$ and arbitrary permutations $\pi$ seems to be very difficult. However, it turns out that we \emph{can} do this when $G$ is a path. To state our main result, we need a bit more terminology.

Let $[k]_q=\frac{1-q^k}{1-q}=1+q+\cdots+q^{k-1}$ and $[k]_q!=[k]_q[k-1]_q\cdots[1]_q$. The \dfn{$q$-binomial coefficient} ${k\brack r}_q$ is the polynomial $\dfrac{[k]_q!}{[r]_q![k-r]_q!}\in\mathbb C[q]$.

Let $X$ be a finite set, and let $f\colon X\to X$ be an invertible map of order $\omega$ (i.e., $\omega$ is the smallest positive integer such that $f^\omega(x)=x$ for all $x\in X$). Let $F(q)\in\mathbb C[q]$ be a polynomial in the variable $q$. Following \cite{CSPDefinition}, we say the triple $(X,f,F(q))$ \dfn{exhibits the cyclic sieving phenomenon} if for every integer $k$, the number of elements of $X$ fixed by $f^k$ is $F(e^{2\pi ik/\omega})$. 

Although we view the set $\mathbb Z/n\mathbb Z$ as a ``cyclic'' object, it will often be convenient to identify 
$\mathbb Z/n\mathbb Z$ with the ``linear'' set $[n]$ and consider the total ordering of its elements given by $1<2<\cdots<n$. If $\pi\colon[n]\to\ZZ/n\ZZ$ is a bijection, then a \dfn{cyclic descent} of $\pi^{-1}$ is an element $i\in\mathbb Z/n\mathbb Z$ such that $\pi^{-1}(i)>\pi^{-1}(i+1)$ (note that $n$ is permitted to be a cyclic descent). 

Let $\Path_n$ denote the path graph with $n$ vertices. In \cite{DefantToric}, the first author conjectured (using different language) that for every bijection 
$\pi\colon[n]\to\mathbb Z/n\mathbb Z$, the order of $\TPro_\pi\colon\Lambda_{\Path_n}\to\Lambda_{\Path_n}$ is $d(n-d)$, where $d$ is the number of cyclic descents of $\pi^{-1}$. Our main theorem not only proves this conjecture, but also determines the entire orbit structure of permutoric promotion in this case. 

\begin{theorem}\label{thm:main}
Let $\pi\colon[n]\to\ZZ/n\ZZ$ be a bijection, and let $d$ be the number of cyclic descents of $\pi^{-1}$. The order of the permutoric promotion operator $\TPro_\pi\colon\Lambda_{\Path_n}\to\Lambda_{\Path_n}$ is $d(n-d)$. Moreover, the triple \[\left(\Lambda_{\Path_n},\TPro_\pi,n(d-1)!(n-d-1)![n-d]_{q^d}{n-1\brack d-1}_{q}\right)\] exhibits the cyclic sieving phenomenon. 
\end{theorem}

Note that when $d=1$, the sieving polynomial in \Cref{thm:main} is $n(n-2)![n-1]_q$, which agrees with \Cref{thm:toric_main}. 

Suppose $B$ is a proper subset of $\mathbb Z/n\mathbb Z$. In order to understand permutoric promotion and prove \Cref{thm:main}, we define the \dfn{broken promotion} operator $\Bro_B\colon\Lambda_G\to\Lambda_G$ as follows. Let $B_1,\ldots,B_k$ be the vertex sets of the connected components of the subgraph of $\Cycle_n$ induced by $B$. For $1\leq i\leq k$, let us write $B_i=\{a(i),a(i)+1,\ldots,b(i)\}$, and let $\Bro_{B_i}=\tau_{b(i)}\cdots\tau_{a(i)+1}\tau_{a(i)}$. We then define $\Bro_B=\Bro_{B_1}\cdots\Bro_{B_k}$ (the order does not matter since $\Bro_{B_1},\ldots,\Bro_{B_k}$ commute with each other). 

Let $\cyc\colon\Lambda_G\to\Lambda_G$ be the \dfn{cyclic shift} operator defined by $(\cyc(\sigma))(v)=\sigma(v)+1$. In \Cref{sec:broken}, we give a description of the operator $\cyc\Bro_B$ in terms of ``gliding globs'' of liquid. Roughly speaking, some of the labels are immersed in globs of liquid, these globs (and their labels) glide along paths in $G$ in a \emph{jeu de taquin} fashion, and then some of the labels are changed appropriately. We also show that certain indicator functions are \emph{homomesic} for $\cyc\Bro_B$ (see \Cref{prop:homomesy}). In \Cref{sec:Broken_Path}, we specialize to the case when $G=\Path_n$ and establish useful connections between broken promotion and permutoric promotion. The purpose of \Cref{sec:divisibility} is to prove that all of the sizes of the orbits of $\TPro_\pi$ are divisible by $\lcm(d,n-d)$ (where $G=\Path_n$ and $d$ is the number of cyclic descents of $\pi^{-1}$). 
In \Cref{sec:Orbit_Structure}, we use this divisibility result to reformulate the analysis of permutoric promotion in terms of ``sliding stones'' and ``colliding coins.'' Roughly speaking, we place some stones on the cycle graph and allow them slide around as we apply toggle operators. At the same time, we place coins on the path graph and allow them to move around and collide with one another. It turns out that the dynamical properties of permutoric promotion are closely related to those of the \emph{stones diagrams} and \emph{coins diagrams}; this allows us to complete the proof of \Cref{thm:main}. Let $\Comp_d(n)$ be the set of compositions of $n$ into $d$ parts, and define $\Rot_{n,d}\colon\Comp_d(n)\to\Comp_d(n)$ by $\Rot_{n,d}(a_1,a_2,\ldots,a_d)=(a_2,\ldots,a_d,a_1)$. We show how to associate an orbit of $\Rot_{n,d}$ to the dynamics of the coins diagrams by recording how far each coin must travel when passing from one collision to the next. It will turn out that the form of the sieving polynomial in \Cref{thm:main} arises from the fact that the triple $\left(\Comp_d(n),\Rot_{n,d},{n-1 \brack d-1}_q\right)$ exhibits the cyclic sieving phenomenon.
In \Cref{sec:orbit_broken}, we apply \Cref{thm:main} to derive the following theorems. 

\begin{theorem}\label{thm:broken_main2}
Let $d$ and $n$ be integers such that $1\leq d\leq n-1$. The order of the operator $\cyc\Bro_{\{1,\ldots,d\}}\colon\Lambda_{\Path_n}\to\Lambda_{\Path_n}$ is $(n-d)n$. Moreover, the triple \[\left(\Lambda_{\Path_n},\cyc\Bro_{\{1,\ldots,d\}},(d-1)!(n-d-1)![n]_{q^{n-d}}[n-d]_{q^d}{n-1\brack d-1}_{q}\right)\] exhibits the cyclic sieving phenomenon. 
\end{theorem}

For any real number $x$, let $[[x]]$ denote the integer closest to $x$, with the convention that $[[x]]=x-1/2$ if $x-1/2\in \ZZ$.

\begin{theorem}\label{thm:broken_main}
Let $d$ and $n$ be positive integers such that $1\leq d\leq\left\lfloor n/2\right\rfloor$. For $i\in\ZZ$, let $s_i=[[in/d]]$, and let $\mathscr R=(\ZZ/n\ZZ)\setminus\{s_1-1,\ldots,s_d-1\}$. The order of the operator $\cyc\Bro_{\mathscr R}\colon\Lambda_{\Path_n}\to\Lambda_{\Path_n}$ is $dn$. Moreover, the triple \[\left(\Lambda_{\Path_n},\cyc\Bro_{\mathscr R},(d-1)!(n-d-1)![n]_{q^{d}}[n-d]_{q^d}{n-1\brack d-1}_{q}\right)\] exhibits the cyclic sieving phenomenon. 
\end{theorem}

While noteworthy on its own, the homomesy result from \Cref{sec:broken} also ends up being useful for proving \Cref{thm:broken_main2,thm:broken_main}.

\section{Basics}\label{sec:basics}

Let $n$ be a positive integer. Given integers $x\leq y$, we let $[x,y]_n$ denote the multiset with elements in $\ZZ/n\ZZ$ obtained by reducing the set $\{x,x+1,\ldots,y\}$ modulo $n$. For example, $[3,7]_3$ is the multiset $\{0,0,1,1,2\}$, where the elements are in $\ZZ/3\ZZ$. 

Given a finite set $X$ and an invertible map $f\colon X\to X$, we let $\Orb_f$ denote the set of orbits of $f$. We will need the following technical lemma concerning the cyclic sieving phenomenon. 

\begin{lemma}\label{lem:CSP_technical}
Let $f\colon X\to X$ and $g\colon \widetilde X\to \widetilde X$ be invertible maps, where $X$ and $\widetilde X$ are finite sets. Let $\{k_i^{m_i}:1\leq i\leq \ell\}$ be the multiset of orbit sizes of $f$, where we use superscripts to denote multiplicities. Let $\omega$ be the order of $f$, and let $F(q)\in\mathbb C[q]$ be such that the triple $(X,f,F(q))$ exhibits the cyclic sieving phenomenon. If $N\in\mathbb Z_{>0}$ and $\chi\in\mathbb Q_{>0}$ are such that $\{(Nk_i)^{\chi m_i}:1\leq i\leq \ell\}$ is the multiset of orbit sizes of $g$, then $g$ has order $N\omega$, and the triple $(\widetilde X,g,\chi[N]_{q^\omega}F(q))$ exhibits the cyclic sieving phenomenon. 
\end{lemma}

\begin{proof}
It is clear that $g$ has order $N\omega$. Fix an integer $k$. When we evaluate the polynomial $\chi[N]_{q^\omega}F(q)$ at $q=e^{2\pi ik/(N\omega)}$, we obtain $\chi[N]_{e^{2\pi ik/N}}F(e^{2\pi i(k/N)/\omega})$; we want to show that this is the number of elements of $\widetilde X$ that are fixed by $g^k$. If $k$ is not divisible by $N$, then there are no such elements because all orbits of $g$ have sizes divisible by $N$; in this case, we are done because the factor $[N]_{e^{2\pi ik/N}}$ is $0$. Now suppose $k$ is divisible by $N$. Because $(X,f,F(q))$ exhibits the cyclic sieving phenomenon, $F(e^{2\pi i(k/N)/\omega})$ is the number of elements of $X$ fixed by $f^{k/N}$. Therefore, $\chi NF(e^{2\pi i(k/N)/\omega})$ is the number of elements of $\widetilde X$ fixed by $g^k$. This completes the proof because $\chi[N]_{e^{2\pi ik/N}}F(e^{2\pi i(k/N)/\omega})=\chi NF(e^{2\pi i(k/N)/\omega})$. 
\end{proof}

We write $\Path_n$ and $\Cycle_n$ for the path with $n$ vertices and the cycle with $n$ vertices, respectively. Identify the vertices of $\Cycle_n$ with $\mathbb Z/n\mathbb Z$ in such a way that they appear in the cyclic order $1,2,\ldots,n$ when read clockwise around the cycle. Let $v_1,\ldots,v_n$ be the vertices of $\Path_n$, listed from left to right (we draw the path horizontally in the plane). As in the introduction, let us fix a graph $G$ and consider the toggle operators $\tau_i$ and the permutoric promotion operators $\TPro_\pi$ on $\Lambda_G$. 

If $\mathcal D$ is an acyclic directed graph with vertex set $\mathcal V$, then we can define a partial order $\leq_{\mathcal D}$ on $\mathcal V$ by declaring $v\leq_{\mathcal D} v'$ whenever there is a directed path in $\mathcal D$ from $v$ to $v'$; the resulting poset $(\mathcal V,\leq_{\mathcal D})$ is called the \dfn{transitive closure} of $\mathcal D$. 

A \dfn{linear extension} of an $n$-element poset $(P,\leq_P)$ is a word $p_1\cdots p_n$ whose letters are the elements of $P$ (with each element appearing exactly once) such that $i\leq j$ whenever $p_i\leq_P p_j$. Given a bijection $\pi\colon[n]\to \mathbb Z/n\mathbb Z$, we obtain an acyclic orientation $\alpha_\pi$ of $\Cycle_n$ by orienting each edge $\{i,i+1\}$ from $i$ to $i+1$ if and only if $\pi^{-1}(i)<\pi^{-1}(i+1)$. If $\beta$ is any acyclic orientation of $\Cycle_n$, then the linear extensions of its transitive closure $(\mathbb Z/n\mathbb Z,\leq_\beta)$ are precisely the words $\pi(1)\cdots\pi(n)$ such that $\pi\colon[n]\to \ZZ/n\ZZ$ is a bijection satisfying $\alpha_\pi=\beta$. 

It is well known that any linear extension of a finite poset can be obtained from any other linear extension of the same poset by repeatedly swapping consecutive incomparable elements. If $i,j\in \mathbb Z/n\mathbb Z$ are incomparable in $(\mathbb Z/n\mathbb Z,\leq_\beta)$, then they are not adjacent in $\Cycle_n$, so the toggle operators $\tau_i$ and $\tau_j$ commute. This implies that if $\pi,\pi'\colon[n]\to\mathbb Z/n\mathbb Z$ are such that $\alpha_\pi=\alpha_{\pi'}$, then the expression for $\TPro_{\pi'}$ as a composition of toggle operators can be obtained from the expression for $\TPro_\pi$ by repeatedly swapping consecutive toggle operators that commute with each other, so $\TPro_\pi=\TPro_{\pi'}$. Therefore, given an acyclic orientation $\beta$ of $\Cycle_n$, it makes sense to write $\TPro_\beta$ for the permutoric promotion operator $\TPro_\pi$, where $\pi\colon[n]\to\mathbb Z/n\mathbb Z$ is any bijection such that $\alpha_\pi=\beta$. 

A \dfn{source} (respectively, \dfn{sink}) of an acyclic orientation is a vertex of in-degree (respectively, out-degree) $0$. If $u$ is a source (respectively, sink), then we can \dfn{flip} $u$ into a sink (respectively, source) by reversing the orientations of all edges incident to $u$. Two acyclic orientations are \dfn{flip equivalent} if one can be obtained from the other by a sequence of flips. 

Let us say two maps $f,g\colon\Lambda_G\to\Lambda_G$ are \dfn{dynamically equivalent} if there is a bijection $\phi\colon\Lambda_G\to\Lambda_G$ such that $f\circ\phi=\phi\circ g$. Note that dynamically equivalent invertible maps have the same orbit structure (that is, they have the same number of orbits of each size). 

\begin{lemma}\label{lem:flip_equivalent}
If $\beta$ and $\beta'$ are acyclic orientations of $\Cycle_n$ that have the same number of edges oriented counterclockwise, then $\TPro_\beta$ and $\TPro_{\beta'}$ are dynamically equivalent. 
\end{lemma}

\begin{proof}
It is known (see \cite{Develin}) that two acyclic orientations of $\Cycle_n$ have the same number of edges oriented counterclockwise if and only if they are flip equivalent. Therefore, we just need to show that if $\beta$ and $\beta'$ are flip equivalent, then $\TPro_\beta$ and $\TPro_{\beta'}$ are dynamically equivalent. It suffices to prove this in the case when $\beta'$ is obtained from $\beta$ by flipping a source $i$ into a sink. In this case, one can check that $\TPro_\beta\circ\tau_i=\tau_i\circ\TPro_{\beta'}$. 
\end{proof}

\begin{lemma}\label{lem:counterclockwise_edges}
Let $\beta$ and $\beta'$ be acyclic orientations of $\Cycle_n$. Let $d$ and $d'$ be the number of edges oriented counterclockwise in $\beta$ and $\beta'$, respectively. If $d=d'$ or $d=n-d'$, then $\TPro_\beta$ and $\TPro_{\beta'}$ are dynamically equivalent. 
\end{lemma}

\begin{proof}
If $d=d'$, then we are done by \Cref{lem:flip_equivalent}. Now suppose $d=n-d'$. Define $\phi\colon\Lambda_G\to\Lambda_G$ by $(\phi(\sigma))(v)=n+1-\sigma(v)$. One can readily check that 
\begin{equation}\label{eq:confusion_phi}
\phi\circ\tau_i=\tau_{n-i}\circ\phi
\end{equation}
for all $i\in\ZZ/n\ZZ$. Let $\pi\colon[n]\to\ZZ/n\ZZ$ be a bijection such that $\alpha_\pi=\beta$, and define $\pi'\colon[n]\to\ZZ/n\ZZ$ by $\pi'(i)=n-\pi(i)$. We have $\TPro_\pi=\TPro_\beta$.  It follows from \eqref{eq:confusion_phi} that $\phi\circ\TPro_{\pi}=\TPro_{\pi'}\circ\phi$. This shows that $\TPro_\pi$ and $\TPro_{\pi'}$ are dynamically equivalent. On the other hand, the number of edges oriented counterclockwise in $\alpha_{\pi'}$ is $n-d$, so it follows from \Cref{lem:flip_equivalent} that $\TPro_{\pi'}$ is dynamically equivalent to $\TPro_{\beta'}$. 
\end{proof}

If $\pi\colon[n]\to\ZZ/n\ZZ$ is a bijection, then the number of cyclic descents of $\pi^{-1}$ is the same as the number of edges oriented counterclockwise in $\alpha_\pi$. This is why cyclic descents appear in \Cref{thm:main}. 

We end this section with a lemma that will allow us to rewrite operators formed as compositions of toggles. We will consider words over the alphabet $\{\tau_1,\ldots,\tau_n\}$ both as words and as permutations of $\Lambda_{G}$. Given such a word $X$, let $X\!\langle i\rangle$ denote the number of occurrences of $\tau_i$ in $X$. 

\begin{lemma}\label{lem:suffix}
Let $\beta$ be an acyclic orientation of $\Cycle_n$. Let $Y$ be a word over the alphabet $\{\tau_1,\ldots,\tau_n\}$ in which each letter appears exactly $k$ times. Suppose that for every suffix $X$ of $Y$ and every arrow $a\to b$ in $\beta$, we have $X\!\langle a\rangle-X\!\langle b\rangle\in\{0,1\}$. When viewed as a bijection from $\Lambda_G$ to itself, $Y$ is equal to $\TPro_\beta^k$. 
\end{lemma}

\begin{proof}
For each $i\in\ZZ/n\ZZ$, consider $k$ formal symbols ${\boldsymbol \tau}_i^{(1)},\ldots,{\boldsymbol \tau}_i^{(k)}$. Let $\mathcal G$ be the group generated by the set ${\bf A}=\left\{{\boldsymbol \tau}_i^{(\ell)}:i\in\ZZ/n\ZZ, \ell\in[k]\right\}$ subject to the relations ${\boldsymbol \tau}_{i}^{(\ell)}{\boldsymbol \tau}_{j}^{(m)}={\boldsymbol \tau}_{j}^{(m)}{\boldsymbol \tau}_{i}^{(\ell)}$ whenever $j\not\in\{i-1,i,i+1\}$. Let $\mathcal D$ be the directed graph with vertex set ${\bf A}$ and with arrows defined as follows: for each arrow $a\to b$ in $\beta$, the graph $\mathcal D$ has arrows ${\boldsymbol\tau}_b^{(\ell)}\to{\boldsymbol\tau}_a^{(\ell)}$ for all $\ell\in[k]$ and ${\boldsymbol\tau}_a^{(m)}\to{\boldsymbol\tau}_b^{(m+1)}$ for all $m\in[k-1]$. Let $({\bf A},\leq_{\mathcal D})$ be the transitive closure of $\mathcal D$. 

Fix a bijection $\pi\colon[n]\to\ZZ/n\ZZ$ such that $\TPro_\pi=\TPro_\beta$, and let $Y'$ be the word $(\tau_{\pi(n)}\cdots\tau_{\pi(1)})^k$ over the alphabet $\{\tau_1,\ldots,\tau_n\}$. Let $Z$ (respectively, $Z'$) be the word over the alphabet $\bf A$ obtained from $Y$ (respectively, $Y'$) by replacing the $\ell$-th occurrence of the letter $\tau_i$ with $\boldsymbol\tau_i^{(\ell)}$. The conditions on $Y$ in the hypothesis of the lemma imply that $Z$ is a linear extension of $({\bf A},\leq_{\mathcal D})$; the word $Y'$ satisfies the same conditions, so $Z'$ is also a linear extension of $({\bf A},\leq_{\mathcal D})$. This means that $Z'$ can be obtained from $Z$ be repeatedly swapping consecutive incomparable elements; each such swap corresponds to one of relations defining $\mathcal G$. Thus, $Z$ and $Z'$ represent the same element of $\mathcal G$. There is a natural homomorphism from $\mathcal G$ to the group of permutations of $\Lambda_G$ that sends each generator ${\boldsymbol \tau}_i^{(\ell)}$ to $\tau_i$. This homomorphism sends $Z$ and $Z'$ to the permutations of $\Lambda_G$ represented by $Y$ and $Y'$, respectively. Hence, these permutations are the same. This completes the proof because the permutation represented by $Y'$ is $\TPro_\beta^k$. 
\end{proof}

\section{Broken Promotion}\label{sec:broken}

In this section, we study the \emph{broken promotion} operators defined in the introduction, describing them in terms of ``gliding globs'' and relating them to permutoric promotion operators. 

\subsection{Jeu de Taquin}
As before, let us fix an $n$-vertex graph $G=(V,E)$. Our arguments in this section will require certain \dfn{jeu de taquin} operators defined as follows. For $i_1,i_2\in\mathbb Z/n\mathbb Z$, define 
$\jdt_{(i_1,i_2)}\colon\Lambda_G\to\Lambda_G$ by \[\jdt_{(i_1,i_2)}(\sigma)=\begin{cases} (i_1\,\,i_2)\circ \sigma & \mbox{if } \{\sigma^{-1}(i_1),\sigma^{-1}(i_2)\}\in E; \\ \sigma & \mbox{if } \{\sigma^{-1}(i_1),\sigma^{-1}(i_2)\}\not\in E.
\end{cases}\] Thus, $\jdt_{(i_1,i_2)}$ has the effect of trying to ``glide'' the label $i_1$ through the label $i_2$; it succeeds in doing so if and only if those labels are on adjacent vertices of $G$. More generally, if $(i_1,\ldots,i_r)$ is a tuple of distinct vertices in $V$, then we define \[\jdt_{(i_1,\ldots,i_r)}=\jdt_{(i_1,i_r)}\jdt_{(i_1,i_{r-1})}\cdots\jdt_{(i_1,i_2)}.\] This operator has the effect of trying to glide $i_1$ through the labels $i_2,\ldots,i_r$ in that order. We will primarily be interested in the case when $(i_1,\ldots,i_r)$ is such that $i_j=i_1+j-1$ for all $1\leq j\leq r$. In this case, $\{i_1,\ldots,i_r\}$ is a cyclic interval $[x,y]_n$, so we simply write $\jdt_{[x,y]_n}$ instead of $\jdt_{(i_1,\ldots,i_r)}$. 

\begin{example}
If $n=6$ and $\sigma$ is the labeling shown on the left in \Cref{Fig1}, then $\jdt_{[5,9]_6}(\sigma)=\jdt_{(5,6,1,2,3)}(\sigma)$ is the labeling shown on the right in \Cref{Fig1}. 
\end{example}

\begin{figure}[ht]
  \begin{center}{\includegraphics[height=1.65cm]{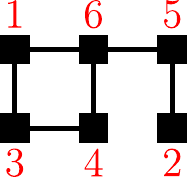}\qquad\qquad\includegraphics[height=1.65cm]{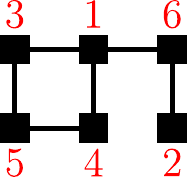}}
  \end{center}
  \caption{On the left is a labeling of a $6$-vertex graph, where the label of each vertex is shown next to it in red. On the right is the labeling $\jdt_{[5,9]_6}(\sigma)=\jdt_{(5,6,1,2,3)}(\sigma)$. The labeling $\jdt_{[5,9]_6}(\sigma)$ was obtained from $\sigma$ by gliding the label $5$ through the labels $6,1,3$.}\label{Fig1}
\end{figure}

\subsection{Broken Promotion}\label{subsec:Description1}
Suppose $B$ is a proper subset of $\mathbb Z/n\mathbb Z$. Recall the definitions of the \emph{broken promotion} operator $\Bro_B\colon\Lambda_G\to\Lambda_G$ and the \emph{cyclic shift} operator $\cyc\colon\Lambda_G\to\Lambda_G$ from \Cref{SecIntro}. We can explicitly describe the action of $\cyc\Bro_B$ on a labeling $\sigma\in\Lambda_G$ as follows. Let $B_1,\ldots,B_k$ be the vertex sets of the connected components of the subgraph of $\Cycle_n$ induced by $B$. For each $1\leq i\leq k$, let $x_i$ and $y_i$ be such that $B_i=[x_i,y_i-1]_n$, and imagine immersing the label $x_i$ in a glob of liquid. The first step is to apply the jeu de taquin operators $\jdt_{[x_i,y_i]_n}$, imagining that the label $x_i$ carries its glob along with it as it glides. For the second step, increase by $1$ the label of each vertex in $\sigma^{-1}((\ZZ/n\ZZ)\setminus\bigcup_{\ell=1}^k[x_\ell,y_\ell]_n)$. If $x_i-1\not\in \bigcup_{\ell=1}^k[x_\ell,y_\ell]_n$, this second step will change the label $x_i-1$ into $x_i$, so there will be two copies of the label $x_i$: one in a glob and the other not in a glob. The third and final step is to change each label $x_i$ that is in a glob to the label $y_i+1$. 

It might not be obvious at first that the procedure described in the preceding paragraph does in fact compute $\cyc\Bro_B(\sigma)$; however, the verification of this fact is straightforward and can be elucidated through examples.

\begin{example}
Suppose $n=9$ and $G=\Path_9$. Let $B=\{1,3,4,7,9\}\subseteq\ZZ/9\ZZ$. The connected components of the subgraph of $\Cycle_9$ induced by $B$ have vertex sets \[B_1=\{3,4\}=[3,4]_9, \quad B_2=\{7\}=[7,7]_9, \quad B_3=\{9,1\}=[9,10]_9.\] Preserving the notation from above, we have $x_1=3$, $y_1=5$, $x_2=7$, $y_2=8$, $x_3=9$, $y_3=11$. Recall that the vertices of $\Path_9$ are $v_1,\ldots,v_9$; let $\sigma\in\Lambda_{\Path_9}$ be the labeling that sends these vertices to $7,1,4,3,5,6,9,2,8$, respectively. \Cref{Fig2} illustrates the three-step procedure for computing $\cyc\Bro_B(\sigma)$, showing that $\cyc\Bro_B(\sigma)$ sends $v_1,\ldots,v_9$ to $9,1,6,4,5,7,2,3,8$, respectively. \hfill $\lozenge$ 
\end{example}

\begin{figure}[ht]
  \[\hphantom{\xrightarrow{\text{Step 2\,\,}}}\begin{array}{l}
  \includegraphics[height=0.686cm]{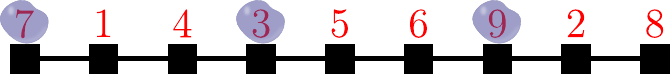}
  \end{array}\xrightarrow{\text{Step 1}}\begin{array}{l}
  \includegraphics[height=0.686cm]{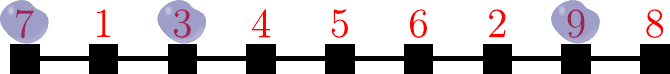}
  \end{array}\] \vspace{.2cm}
  \[\xrightarrow{\text{Step 2}}\begin{array}{l}
  \includegraphics[height=0.686cm]{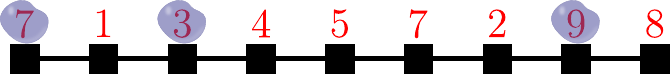}
  \end{array}\xrightarrow{\text{Step 3}}\begin{array}{l}
  \includegraphics[height=0.686cm]{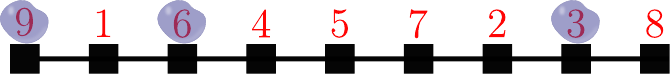}
  \end{array}\]\caption{The three steps for applying $\cyc\Bro_B$, where $B=\{1,3,4,7,9\}\subseteq\ZZ/9\ZZ$.}\label{Fig2}
\end{figure}

\subsection{Broken Promotion for the Complement of an Independent Set}\label{subsec:Description2}

Suppose $1\leq d\leq \left\lfloor n/2\right\rfloor$, and let $\cdots<s_{-1}<s_0<s_1<s_2<\cdots$ be a bi-infinite sequence of integers such that $s_{i+d}=s_i+n$ and $s_{i+1}\geq s_i+2$ for all $i\in\ZZ$. Then $\mathscr S=\{s_1,\ldots,s_d\}$ is an independent set of size $d$ in $\Cycle_n$. Let $\beta_{\mathscr S}$ be the acylic orientation of $\Cycle_n$ in which the elements of $\mathscr S$ are sources and all edges not incident to elements of $\mathscr S$ are oriented clockwise. The sinks of $\beta_{\mathscr S}$ are the elements of $\mathscr S-1:=\{s_1-1,\ldots,s_d-1\}$. Let us write $\mathscr R=(\ZZ/n\ZZ)\setminus(\mathscr S-1)$. 

In \Cref{subsec:Description1}, we gave a three-step description of the action of $\cyc\Bro_B$ on a labeling $\sigma$ when $B$ is an arbitrary proper subset of $\ZZ/n\ZZ$. This description simplifies when $B$ is $\mathscr R$ because in this case, we have $x_i=s_i$ and $y_i=s_{i+1}-1$, so $\bigcup_{\ell=1}^k[x_\ell,y_\ell]_n$ is all of $\ZZ/n\ZZ$ (so the second step in the earlier description has no effect). Hence, we have the following simpler two-step procedure. Immerse each label $s_i$ in a glob of liquid. The first step is to apply the jeu de taquin operators $\jdt_{[s_i,s_{i+1}-1]_n}$ (for $1\leq i\leq d$), imagining that the label $s_i$ carries its glob with it as it glides. The second step is to cyclically rotate the labels in the globs, changing each label $s_i$ to $s_{i+1}$ (modulo $n$).

\begin{example}
Suppose $n=9$ and $d=3$. Let $s_1=3$, $s_2=7$, $s_3=9$. Then $\mathscr S=\{3,7,9\}$, $\mathscr S-1=\{2,6,8\}$, and $\mathscr R=(\ZZ/9\ZZ)\setminus(\mathscr S-1)=\{1,3,4,5,7,9\}$. The first step in the above two-step procedure for applying $\cyc\Bro_{\mathscr R}$ is to immerse $3$, $7$, and $9$ in globs of liquid and apply $\jdt_{[3,6]_9}$, $\jdt_{[7,8]_9}$, and $\jdt_{[9,11]_9}$. The second step is to cyclically rotate the labels $3,7,9$. This is illustrated in \Cref{Fig3}. \hfill $\lozenge$
\end{example}

\begin{figure}[ht]
  \[\begin{array}{l}
  \includegraphics[height=0.686cm]{PermutoricPIC4}
  \end{array}\xrightarrow{\text{Step 1}}\begin{array}{l}
  \includegraphics[height=0.686cm]{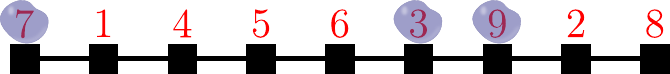}
  \end{array}\] \vspace{.2cm}
  \[\hphantom{\begin{array}{l}
  \includegraphics[height=0.686cm]{PermutoricPIC5}
  \end{array}}\xrightarrow{\text{Step 2}}\begin{array}{l}
  \includegraphics[height=0.686cm]{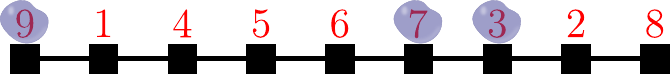}
  \end{array}\]\caption{The two steps for applying $\cyc\Bro_{\mathscr R}$, where $\mathscr R=\{1,3,4,5,7,9\}\subseteq\ZZ/9\ZZ$.}\label{Fig3}
\end{figure}

\begin{remark}\label{Rem:SameOrder}
Suppose $G=\Path_n$. Neither of the two steps in the above procedure change the relative order in which the labels in $(\ZZ/n\ZZ)\setminus\mathscr S$ (i.e., the labels not in the globs) appear from left to right along the path. For example, in \Cref{Fig3}, the labels in $(\ZZ/n\ZZ)\setminus\mathscr S$ are $1,2,4,5,6,8$. In every step of the procedure, these labels appear in the order $1,4,5,6,2,8$. \hfill $\triangle$
\end{remark}

\subsection{Permutoric Promotion and Broken Promotion}

The following lemma relates the operators $\cyc\Bro_{\mathscr R}$ and $\TPro_\beta$; it will later allow us to use the description of $\cyc\Bro_{\mathscr R}$ in terms of gliding globs given in \Cref{subsec:Description2} to gain a better understanding of permutoric promotion. 

\begin{lemma}\label{lem:commutes}
Let $\mathscr{S}$ be a $d$-element independent set in $\Cycle_n$, and let $\mathscr R=(\mathbb Z/n\mathbb Z)\setminus(\mathscr S-1)$. Then \[\cyc\Bro_{\mathscr R}\TPro_{\beta_{\mathscr S}}=\TPro_{\beta_{\mathscr S}}\cyc\Bro_{\mathscr R}.\]
\end{lemma}

\begin{proof}
Preserve the notation from \Cref{subsec:Description2}. Let $\mathscr B=\mathscr R\setminus\mathscr S$. The vertex sets of the connected components of the subgraph of $\Cycle_n$ induced by $\mathscr R$ are $[s_1,s_2-2]_n,[s_2,s_3-2]_n,\ldots,[s_d,s_{d+1}-2]_n$, so \[\Bro_{\mathscr R}=\prod_{i=1}^d\Bro_{[s_i,s_{i+1}-2]_n}=\prod_{i=1}^d(\Bro_{[s_i+1,s_{i+1}-2]_n}\tau_{s_i})=\prod_{i=1}^d\Bro_{[s_i+1,s_{i+1}-2]_n}\prod_{i=1}^d\tau_{s_i}=\Bro_{\mathscr{B}}\Bro_{\mathscr S}.\] A similar computation shows that \[\cyc\Bro_{\mathscr R} \cyc^{-1}=\prod_{i=1}^d\Bro_{[s_i+1,s_{i+1}-1]_n}=\Bro_{\mathscr S-1}\Bro_{\mathscr B}.\] We also have $\cyc\Bro_{\mathscr S-1}\cyc^{-1}=\Bro_{\mathscr S}$. Therefore, 
\begin{align*}
\cyc\Bro_{\mathscr R}\Bro_{\mathscr S-1}&=(\cyc\Bro_{\mathscr R} \cyc^{-1})(\cyc\Bro_{\mathscr S-1}\cyc^{-1})\cyc \\
&=\Bro_{\mathscr S-1}(\Bro_{\mathscr B}\Bro_{\mathscr S}) \cyc \\ 
&=\Bro_{\mathscr S-1}\Bro_{\mathscr R}\cyc.
\end{align*} This shows that \[\cyc\Bro_{\mathscr R}\Bro_{\mathscr S-1}\Bro_{\mathscr R}=\Bro_{\mathscr S-1}\Bro_{\mathscr R}\cyc\Bro_{\mathscr R}.\] The desired result now follows from the observation that $\TPro_{\beta_{\mathscr S}}=\Bro_{\mathscr S-1}\Bro_{\mathscr R}$. 
\end{proof}

\subsection{Homomesy}\label{subsec:homomesy}
We end this section with a theorem about broken promotion that will be useful in \Cref{sec:orbit_broken} but that we also believe is interesting in its own right. This proposition concerns the notion of \emph{homomesy}, which Propp and Roby introduced in 2013 \cite{ProppRobyFPSAC,ProppRoby}; it is now one of the central focuses in dynamical algebraic combinatorics. Suppose $X$ is a finite set and $f\colon X\to X$ is an invertible map. Let $\Orb_f$ denote the set of orbits of $f$. A \dfn{statistic} on $X$ is a function $\text{stat}\colon X\to\mathbb R$. We say the statistic $\text{stat}$ is \dfn{homomesic} for $f$ with average $a$ if $\frac{1}{|\mathcal O|}\sum_{x\in\mathcal O}\text{stat}(x)=a$ for every orbit $\mathcal O\in\Orb_f$. 

\begin{proposition}\label{prop:homomesy}
Suppose $G$ is connected. Let $v$ be a vertex of $G$, and let $i\in\ZZ/n\ZZ$. Define $\mathbbm{1}_{v,i}\colon\Lambda_G\to\mathbb R$ by \[\mathbbm{1}_{v,i}(\sigma)=\begin{cases} 1 & \mbox{if } \sigma(v)=i; \\ 0 & \mbox{if } \sigma(v)\neq i. \end{cases}\] If $B\subseteq\ZZ/n\ZZ$ and $i-1\not\in B$, then $\mathbbm{1}_{v,i}$ is homomesic for the map $\cyc\Bro_B$ with average $1/n$. 
\end{proposition}

\begin{proof}
By symmetry, it suffices to prove the result when $i=1$. We identify $\ZZ/n\ZZ$ with $[n]$ and consider the total ordering $1<2<\dots<n$. Given a labeling $\sigma\in\Lambda_G$, we obtain an acyclic orientation $\eta_\sigma$ by orienting each edge $\{x,y\}$ from $x$ to $y$ if and only if $\sigma(x)<\sigma(y)$. Observe that $\eta_{\tau_j(\sigma)}=\eta_\sigma$ for every $j\in[n-1]$ and $\sigma\in\Lambda_G$; since $n=i-1\not\in B$, this implies that $\eta_{\Bro_B(\sigma)}=\eta_\sigma$ for every $\sigma\in\Lambda_G$. It is also straightforward to see that $\eta_{\cyc(\sigma)}$ is obtained from $\eta_\sigma$ by flipping the vertex $(\cyc(\sigma))^{-1}(1)$ from a sink to a source. Therefore, $\eta_{\cyc(\Bro_B(\sigma))}$ is obtained from $\eta_\sigma$ by flipping $(\cyc(\Bro_B(\sigma)))^{-1}(1)$ from a sink to a source. 

Let $\mathcal O$ be an orbit of $\cyc\Bro_B$, and fix $\mu_0\in\mathcal O$. Let $\mu_t=(\cyc\Bro_B)^t(\mu_0)$ for all $t\in\ZZ$. Consider an edge $\{x_0,y_0\}$ in $G$. Let $\cdots<k(0)<k(1)<k(2)<\cdots$ be the integers such that $\mu_{k(j)}^{-1}(1)\in\{x,y\}$. Without loss of generality, assume $x_0\to y_0$ is an arrow in $\eta_{\mu_{k(0)}}$. According to the previous paragraph, the orientations of $\{x_0,y_0\}$ are different in $\eta_{\mu_{t-1}}$ and $\eta_{\mu_t}$ if and only if $t\in\{\ldots,k(0),k(1),k(2),\ldots\}$. Moreover, we have $\mu_{k(j)}(x_0)=1$ if $x_0\to y_0$ is an arrow in $\eta_{\mu_{k(j)}}$, whereas $\mu_{k(j)}(y_0)=1$ if $y_0\to x_0$ is an arrow in $\eta_{\mu_{k(j)}}$. It follows that for $t\in\ZZ$, we have $\mu_t(x_0)=1$ if and only if $t=k(j)$ for some even $j$; similarly, $\mu_t(y_0)=1$ if and only if $t=k(j)$ for some odd $j$. This shows that the number of labelings in $\mathcal O$ that send $x_0$ to $1$ is the same as the number of labelings in $\mathcal O$ that send $y_0$ to $1$. Because the edge $\{x_0,y_0\}$ was arbitrary and $G$ is connected, it follows that for any two vertices $x$ and $y$ of $G$, the number of labelings in $\mathcal O$ that send $x$ to $1$ is the same as the number of labelings in $\mathcal O$ that send $y$ to $1$. This implies the desired result. 
\end{proof}

\begin{example}
Suppose $n=5$ and $B=\{1,3,4\}$. \Cref{Fig4} depicts an orbit of $\cyc\Bro_{B}$ for a particular choice of a graph $G$. Select an arbitrary vertex $v$ of $G$. As predicted by \Cref{prop:homomesy}, exactly $1$ of the $5$ labelings in this orbit sends $v$ to $1$, and exactly $1$ of the $5$ labelings in the orbit sends $v$ to $3$. \hfill $\lozenge$   
\end{example}

\begin{figure}[ht]
  \begin{center}{\includegraphics[height=4.7cm]{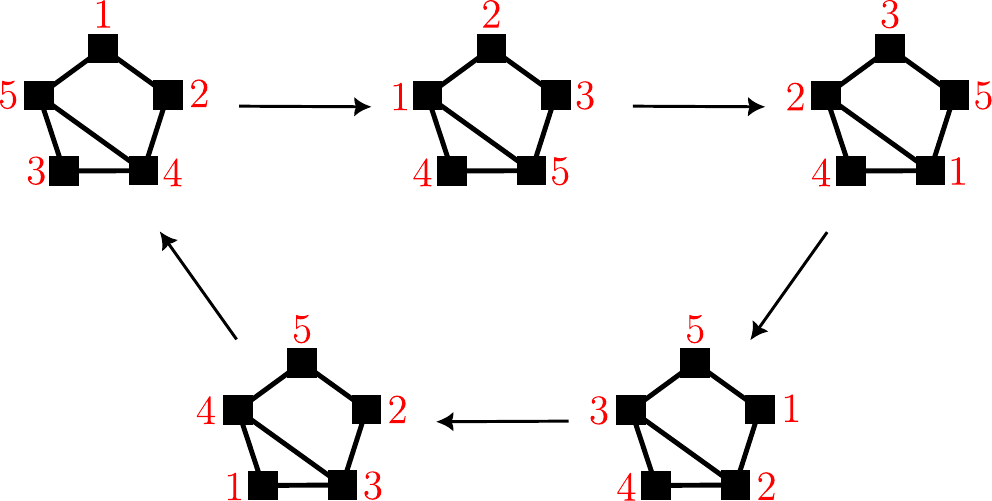}}
  \end{center}
  \caption{An orbit of $\cyc\Bro_{\{1,3,4\}}$ for a particular graph $G$. Note that each of the labels $1$ and $3$ appears on each vertex exactly once throughout the orbit. }\label{Fig4}
\end{figure}

\section{Broken Promotion on a Path}\label{sec:Broken_Path}
Throughout the rest of the article, we will specialize to the case when $G=\Path_n$ is the path with $n$ vertices. 

Suppose $1\leq d\leq\lfloor n/2\rfloor$. In \Cref{sec:broken}, we considered an arbitrary bi-infinite sequence \[\cdots<s_{-1}<s_0<s_1<s_2<\cdots\] satisfying $s_{i+d}=s_i+n$ and $s_{i+1}\geq s_i+2$ for all $i\in\ZZ$. In this section, we specialize our attention to a particular sequence. We write $[[x]]$ for the integer closest to a real number $x$, with the convention that $[[x]]=x-1/2$ if $x-1/2\in \ZZ$. For $i\in\ZZ$, let $s_i=[[in/d]]$. As in \Cref{sec:broken}, we let $\mathscr S$ be the independent set $\{s_1,\ldots,s_d\}$ of $\Cycle_n$ and set $\mathscr R=(\ZZ/n\ZZ)\setminus(\mathscr S-1)$. Let $\beta=\beta_{\mathscr S}$ be the acyclic orientation of $\Cycle_n$ whose sources are the elements of $\mathscr S$ and whose sinks are the elements of $\mathscr S-1$. Then $\beta$ has exactly $d$ counterclockwise edges. 

In what follows, when we consider the size of the intersection of a multiset with a set, we count the elements according to their multiplicity in the multiset. For example, in the next proposition, $|[j-q,j-1]_n\cap(\mathscr S-1)|$ should be interpreted as the number of elements of $[j-q,j-1]_n$, counted with multiplicity, that are also elements of the set $\mathscr S-1$. 

\begin{proposition}\label{prop:TProcycBro_R}
Let $\gamma,q,r$ be nonnegative integers such that $0\leq r\leq n-d-1$ and $\gamma n=q(n-d)+r$. Let $J=\{j\in[n]:q-\gamma+1\leq |[j-q,j-1]_n\cap(\mathscr S-1)|\}$. Then $|J|=r$, and $\TPro_\beta^\gamma=\cyc^{-q}\Bro_J(\cyc\Bro_{\mathscr R})^q$.
\end{proposition}

\Cref{prop:TProcycBro_R} will be crucial in the next section when we prove that the sizes of the orbits of $\TPro_\beta$ are all divisible by $\lcm(d,n-d)$. Before proving this proposition, we need a technical lemma.

\begin{lemma}\label{Lem:R1}
Let $\gamma,q,r,J$ be as in \Cref{prop:TProcycBro_R}. We have $J\cap(\mathscr S-1)=\emptyset$. Also, if $i\in\mathscr R$ and $i+1\in J$, then $i\in J$. 
\end{lemma}

\begin{proof}
Consider some $s_j-1\in\mathscr S-1$. Recall that $s_j=[[jn/d]]$. It is straightforward to check that \[|[s_j-1-q,s_j-2]_n\cap(\mathscr S-1)|=|[s_j-1-q,s_j-1]_n\cap(\mathscr S-1)|-1\leq \frac{(q+1)d}{n}.\] Because $r\leq n-d-1$, we have
\[|[s_j-1-q,s_j-2]_n\cap(\mathscr S-1)|<\frac{(q+1)d}{n}+\frac{n-d-r}{n}=q-\frac{q(n-d)+r}{n}+1=q-\gamma+1,\] so $s_j-1\not\in J$. This proves that $J\cap(\mathscr S-1)=\emptyset$. 

Now suppose $i\in\mathscr R=(\ZZ/n\ZZ)\setminus(\mathscr S-1)$ and $i+1\in J$. We have \[|[i-q,i-1]_n\cap(\mathscr S-1)|\geq |[(i+1)-q,(i+1)-1]_n\cap(\mathscr S-1)|\geq q-\gamma+1,\] so $i\in J$. 
\end{proof}

\begin{proof}[Proof of \Cref{prop:TProcycBro_R}]
As in \Cref{lem:suffix}, we will consider words over the alphabet $\{\tau_1,\ldots,\tau_n\}$ both as words and as permutations of $\Lambda_{\Path_n}$. Given such a word $X$, recall that we write $X\!\langle i\rangle$ for the number of occurrences of $\tau_i$ in $X$. 

Let $i_1,\ldots,i_{n-d}$ be an ordering of the elements of $\mathscr R$ such that $\Bro_{\mathscr R}=\tau_{i_{n-d}}\cdots\tau_{i_1}$. Consider the word \[W=\tau_{i_{n-d}-(q-1)}\cdots\tau_{i_1-(q-1)}\tau_{i_{n-d}-(q-2)}\cdots\tau_{i_1-(q-2)}\cdots\tau_{i_{n-d}-1}\cdots\tau_{i_1-1}\tau_{i_{n-d}}\cdots\tau_{i_1}.\] For each $i\in\ZZ/n\ZZ$, we have $W\!\langle i\rangle=q-|[i,i+q-1]_n\cap(\mathscr S-1)|$. The size of \[[i,i+q-1]_n\cap(\mathscr S-1)=[i,i+q-1]_n\cap\{[[n/d]]-1,[[2n/d]]-1,\ldots,[[dn/d]]-1\}\] must be $\left\lfloor qd/n\right\rfloor$ or $\left\lceil qd/n\right\rceil$. Using the identity $q(n-d)=\gamma n-r$, we find that $W\!\langle i\rangle\in\{\gamma-1,\gamma\}$ for all $i\in\ZZ/n\ZZ$. The total number of toggle operators in $W$ is $q(n-d)=\gamma n-r$, so there are exactly $r$ elements $i\in\ZZ/n\ZZ$ such that $W\!\langle i\rangle=\gamma-1$. Furthermore, we have $W\!\langle i\rangle=\gamma-1$ if and only if $i+q\in J$. This proves that $|J|=r$ and that $q-|[j-q,j-1]_n\cap(\mathscr S-1)|=\gamma-1$ for every $j\in J$.

It follows from \Cref{Lem:R1} that we can choose the ordering $i_1,\ldots,i_{n-d}$ so that $\Bro_J=\tau_{i_r}\cdots\tau_{i_1}$. For each $k\in\mathbb Z$, we have $\cyc^{-k}\Bro_{\mathscr R}\cyc^k=\tau_{i_{n-d}-k}\cdots\tau_{i_1-k}$. Thus, when we view $W$ as a permutation of $\Lambda_{\Path_n}$, it is equal to \[(\cyc^{-(q-1)}\Bro_{\mathscr R}\cyc^{q-1})\cdots (\cyc^{-1}\Bro_{\mathscr R}\cyc)\Bro_{\mathscr R}=\cyc^{-q}(\cyc\Bro_{\mathscr R})^q.\] When we view the word $W'=\tau_{i_r-q}\cdots\tau_{i_1-q}$ as a permutation, it is equal to $\cyc^{-q}\Bro_J \cyc^q$, so $W'W=\cyc^{-q}\Bro_J (\cyc\Bro_{\mathscr R})^q$. Every toggle operator $\tau_i$ occurs exactly $\gamma$ times in the word $W'W$. Our goal is to prove that the permutation $W'W$ of $\Lambda_{\Path_n}$ is equal to $\TPro_\beta^\gamma$. Setting $Y=W'W$ in \Cref{lem:suffix}, we find that it suffices to show that if $X$ is a suffix of $W'W$ and $a\to b$ is an arrow in $\beta$, then $X\!\langle a\rangle-X\!\langle b\rangle\in\{0,1\}$. 

Given $A\subseteq\ZZ/n\ZZ$ and $i\in\ZZ/n\ZZ$, let \[A(i)=\begin{cases} 1 & \mbox{if } i\in A; \\ 0 & \mbox{if } i\not\in A. \end{cases}\] Let $X$ be a suffix of $W'W$, and write $|X|=k(n-d)+m$ for some nonnegative integers $k$ and $m$ with $0\leq m\leq n-d-1$. Then \[X=\tau_{i_m-k}\cdots\tau_{i_1-k}\tau_{i_{n-d}-(k-1)}\cdots\tau_{i_1-(k-1)}\cdots\tau_{i_{n-d}-1}\cdots\tau_{i_1-1}\tau_{i_{n-d}}\cdots\tau_{i_1}.\] Let $Q=\{i_1,\ldots,i_m\}\subseteq\mathscr R$. Let $a\to b$ be an arrow in $\beta$; we want to show that $X\!\langle a\rangle-X\!\langle b\rangle\in\{0,1\}$. To do this, let us first prove that 
\begin{equation}\label{Eq:SQ}
(\mathscr S-1)(\ell)+Q(\ell)-Q(\ell+1)\in \{0,1\}\text{ for all }\ell\in\ZZ/n\ZZ.
\end{equation}
Because $(\mathscr S-1)\cap Q=\emptyset$, we must have $(\mathscr S-1)(\ell)+Q(\ell)-Q(\ell+1)\leq 1$. Suppose by way of contradiction that $(\mathscr S-1)(\ell)+Q(\ell)-Q(\ell+1)<0$. Then $(\mathscr S-1)(\ell)=Q(\ell)=0$ and $Q(\ell+1)=1$. This implies that $\ell+1$ is not in the set $\mathscr S$ of sources of $\beta$, so there is an arrow $\ell\to \ell+1$ in $\beta$. Hence, $\ell$ appears before $\ell+1$ in the ordering $i_1,\ldots,i_{n-d}$. Since $\ell+1\in Q=\{i_1,\ldots,i_m\}$, this forces $\ell\in Q$, which is a contradiction. 

We can now prove that $X\!\langle a\rangle-X\!\langle b\rangle\in\{0,1\}$; we consider two cases. 

\medskip 

\noindent{\bf Case 1.} Suppose $b=a+1$. In this case, $X\!\langle a\rangle=k-|[a,a+k-1]_n\cap(\mathscr S-1)|+Q(a+k)$ and $X\!\langle b\rangle=k-|[a+1,a+k]_n\cap(\mathscr S-1)|+Q(a+k+1)$, so \[X\!\langle a\rangle - X\!\langle b\rangle=-(\mathscr S-1)(a)+(\mathscr S-1)(a+k)+Q(a+k)-Q(a+k+1).\] Because $a\to a+1$ is an arrow in $\beta$, we know that $(\mathscr S-1)(a)=0$. If we set $\ell=a+k$ in \eqref{Eq:SQ}, we find that $X\!\langle a\rangle-X\!\langle b\rangle=(\mathscr S-1)(a+k)+Q(a+k)-Q(a+k+1)\in\{0,1\}$. 

\medskip 

\noindent {\bf Case 2.} Suppose $b=a-1$. In this case, $X\!\langle a\rangle=k-|[a,a+k-1]_n\cap(\mathscr S-1)|+Q(a+k)$ and $X\!\langle b\rangle=k-|[a-1,a+k-2]_n\cap(\mathscr S-1)|+Q(a+k-1)$, so \[X\!\langle a\rangle - X\!\langle b\rangle=(\mathscr S-1)(a-1)-(\mathscr S-1)(a+k-1)+Q(a+k)-Q(a+k-1).\]
Because $a\to a-1$ is an arrow in $\beta$, it follows from the definition of $\beta$ that $(\mathscr S-1)(a-1)=1$. If we set $\ell=a+k-1$ in \eqref{Eq:SQ}, we find that $(\mathscr S-1)(a+k-1)+Q(a+k-1)-Q(a+k)\in\{0,1\}$. Therefore, $X\!\langle a\rangle-X\!\langle b\rangle =1-((\mathscr S-1)(a+k-1)+Q(a+k-1)-Q(a+k))\in\{0,1\}$.
\end{proof}

\section{Divisibility of Permutoric Promotion Orbit Sizes}\label{sec:divisibility}

Our goal in this section is to prove the following proposition. 
\begin{proposition}\label{prop:divisibility}
If $\beta$ is an acyclic orientation of $\Cycle_n$ with $d$ counterclockwise edges, then every orbit of $\TPro_\beta$ has size divisible by $\lcm(d,n-d)$. 
\end{proposition}

\Cref{lem:counterclockwise_edges} tells us that it suffices to prove \Cref{prop:divisibility} when $1\leq d\leq\left\lfloor n/2\right\rfloor$. Furthermore, if $d=1$, then $\TPro_\beta$ is dynamically equivalent to the toric promotion operator $\TPro$, so it follows from \Cref{thm:toric_main} (specialized to the case when $G=\Path_n$) that all orbits of $\TPro_\beta$ have size $n-1$. Thus, we may assume in what follows that $2\leq d\leq\left\lfloor n/2\right\rfloor$. By \Cref{lem:counterclockwise_edges}, we only need to prove \Cref{prop:divisibility} for one specific choice of an acyclic orientation $\beta$ with $d$ counterclockwise edges. As in \Cref{sec:Broken_Path}, let $s_i=[[in/d]]$, and let $\mathscr S$ be the independent set $\{s_1,\ldots,s_d\}$ of $\Cycle_n$. Let $\mathscr R=(\ZZ/n\ZZ)\setminus(\mathscr S-1)$. Let $\beta=\beta_{\mathscr S}$ be the acyclic orientation of $\Cycle_n$ whose sources are the elements of $\mathscr S$ and whose sinks are the elements of $\mathscr S-1$. We will prove that every orbit of $\TPro_\beta$ has size divisible by $\lcm(d,n-d)$. 

Fix a labeling $\lambda\in\Lambda_{\Path_n}$. Let $\gamma$ be the size of the orbit of $\TPro_\beta$ containing $\lambda$. Using the division algorithm, we can write $\gamma n=q(n-d)+r$, where $q$ and $r$ are nonnegative integers and $0\leq r\leq n-d-1$. As in \Cref{prop:TProcycBro_R}, let
\[J=\{j\in[n]:q-\gamma+1\leq |[j-q,j-1]_n\cap(\mathscr S-1)|\}.\]

Since \Cref{prop:TProcycBro_R} allows us to rewrite $\TPro_\beta^\gamma$ in terms of the operator $\cyc\Bro_{\mathscr R}$, we will want to consider the orbit of $\lambda$ under $\cyc\Bro_{\mathscr R}$. Thus, we let \[\mathcal M=\{(\cyc\Bro_{\mathscr R})^t(\lambda):t\in\ZZ\}.\] In \Cref{subsec:Description2}, we described how to compute the action of $\cyc\Bro_{\mathscr R}$ on a labeling via a two-step procedure involving gliding globs. As mentioned in \Cref{Rem:SameOrder}, neither of the two steps in this procedure change the relative order in which the labels in $(\ZZ/n\ZZ)\setminus\mathscr S$ appear along the path. Thus, we have the following lemma. 

\begin{lemma}\label{Lem:SameOrder}
For every $\mu\in\mathcal M$, the order in which the labels in $(\ZZ/n\ZZ)\setminus \mathscr S$ appear along the path in $\mu$ is the same as the order in which they appear along the path in $\lambda$. 
\end{lemma}

We are now in a position to prove that $\gamma$ is divisible by $n-d$. 

\begin{lemma}\label{lem:divisible_n-d}
If $\lambda\in\Lambda_{\Path_n}$ belongs to an orbit of $\TPro_\beta$ of size $\gamma$, then $\gamma$ is divisible by $n-d$. 
\end{lemma}

\begin{proof}
Recall that we write $\gamma n=q(n-d)+r$ using the division algorithm. The map $\TPro_\beta^\gamma$
fixes $\lambda$. \Cref{lem:commutes} tells us that $\TPro_\beta$ commutes with $\cyc\Bro_{\mathscr R}$, so $\TPro_\beta^\gamma$ acts as the identity on $\mathcal M$ and thus trivially restricts to a bijection from $\mathcal M$ to itself. Since $\TPro_\beta^\gamma=\cyc^{-q}\Bro_J(\cyc\Bro_{\mathscr R})^q$ by 
\Cref{prop:TProcycBro_R}, the map $\cyc^{-q}\Bro_J$ also restricts to a bijection from $\mathcal M$ to itself. Let $u_1,\ldots,u_{n-d}$ be the elements of $(\ZZ/n\ZZ)\setminus\mathscr S$, listed in the order in which they appear from left to right along the path in $\lambda$. It follows from \Cref{Lem:R1} that there exist integers $\ldots,y_0,y_1,y_2,\ldots$ satisfying $y_{i+d}=y_i+n$ and $s_i\leq y_i\leq s_{i+1}-1$ for all $i$ such that $J=\bigcup_{i=1}^d[s_i,y_i-1]_n$ (viewing $J$ as a subset of $\ZZ/n\ZZ$). For each $1\leq i\leq d$, we have that $y_i\not\in J$ and $y_i-1\in J$, so it follows from the definition of $J$ that $|[y_i-q,y_i-1]\cap(\mathscr S-1)|<|[y_i-1-q,y_i-2]\cap(\mathscr S-1)|$. We deduce that $y_i-1-q\in \mathscr S-1$ for all $1\leq i\leq d$. Therefore, 
\begin{equation}\label{eq:yS}
\mathscr S=\{y_1-q,\ldots,y_d-q\}. 
\end{equation} 

Let \[\zeta=\begin{cases} 0 & \mbox{if } u_1\in\bigcup_{\ell=1}^d[s_\ell,y_\ell]_n; \\ 1 & \mbox{otherwise.} \end{cases}\] Note that, regardless of the value of $\zeta$, the element $u_1+\zeta-1$ cannot be of the form $y_i$ for any integer $i$. Therefore, it follows from \eqref{eq:yS} that 
\begin{equation}\label{eq:zeta}
u_1+\zeta-q-1\not\in\mathscr S.
\end{equation} 
As mentioned above, $\cyc^{-q}\Bro_J$ restricts to a bijection from $\mathcal M$ to itself; thus, it follows from \Cref{Lem:SameOrder} that the labels in $(\ZZ/n\ZZ)\setminus\mathscr S$ appear in the order $u_1,\ldots,u_{n-d}$ from left to right along the path in the labeling $\cyc^{-q}\Bro_J(\lambda)$. 

Consider applying $\cyc\Bro_J$ to $\lambda$ using the three-step gliding-globs procedure described in \Cref{sec:broken}. We immerse the labels $s_1,\ldots,s_d$ and then apply the jeu de taquin operators $\jdt_{[s_i,y_i]_n}$, imagining that the label $s_i$ carries its glob along with it as it glides. After this initial step, the label $u_1$ will be on some vertex $z$; at this point in time, all of the vertices to the left of $z$ have globs of liquid on them, while $z$ does not. We claim that $\lambda(z)\in\bigcup_{\ell=1}^d[s_\ell,y_\ell]_n$ if and only if $u_1\in\bigcup_{\ell=1}^d[s_\ell,y_\ell]_n$. This is obvious if $\lambda(z)=u_1$. On the other hand, if $\lambda(z)\neq u_1$, then it follows from the definition of the jeu de taquin operators that $\lambda(z)$ and $u_1$ must both be in $\bigcup_{\ell=1}^d[s_\ell,y_\ell]_n$. This proves the claim, which is equivalent to the statement that $\zeta=1$ if and only if $z\in\lambda^{-1}((\ZZ/n\ZZ)\setminus\bigcup_{\ell=1}^d[s_\ell,y_\ell]_n)$. The second step in the gliding-globs procedure increases by $1$ the label of each vertex in $\lambda^{-1}((\ZZ/n\ZZ)\setminus\bigcup_{\ell=1}^d[s_\ell,y_\ell]_n)$; therefore, the label of $z$ is $u_1+\zeta$ after the second step. Note that the second step does not move any of the globs of liquid. The third step of the procedure changes the label in each glob of liquid to a label of the form $y_i+1$. It follows that in the labeling $\cyc\Bro_J(\lambda)$, the labels of the vertices to the left of $z$ are all of the form $y_i+1$, and the label of $z$ is $u_1+\zeta$. This means that in the labeling $\cyc^{-q}\Bro_J(\lambda)=\cyc^{-q-1}(\cyc\Bro_J(\lambda))$, the labels of the vertices to the left of $z$ are all of the form $y_i-q$ (i.e., they are in $\mathscr S$ by \eqref{eq:yS}), and the label of $z$ is $u_1+\zeta-q-1$. Combining this with \eqref{eq:zeta}, we find that $u_1+\zeta-q-1$ is the label in $(\ZZ/n\ZZ)\setminus\mathscr S$ that appears farthest to the left in the labeling $\cyc^{-q}\Bro_J(\lambda)$. As mentioned above, $\cyc^{-q}\Bro_J$ sends $\mathcal M$ to itself, so it follows from \Cref{Lem:SameOrder} that the labels in $(\ZZ/n\ZZ)\setminus\mathscr S$ appear in the order $u_1,\ldots,u_{n-d}$ in $\cyc^{-q}\Bro_J(\lambda)$. Consequently, $u_1+\zeta-q-1=u_1$. This proves that $q$ is congruent to $0$ or $-1$ modulo $n$.  

We defined $q$ and $r$ so that $\gamma n=q(n-d)+r$ and $0\leq r\leq n-d-1$. This implies that $r\not\equiv -d\pmod n$. Reading the first equation modulo $n$ yields $r\equiv qd\pmod n$, so $q\not\equiv -1\pmod n$. Therefore, we must have $q\equiv 0\pmod n$ and $r=0$. Writing $q=mn$, we find that $\gamma=m(n-d)$, which completes the proof. 
\end{proof}

Finally, we can complete the proof of the main result of this section. 

\begin{proof}[Proof of \Cref{prop:divisibility}]
As discussed at the beginning of this section, it suffices to prove that every orbit of $\TPro_\beta$ is divisible by $\lcm(d,n-d)$, where $\beta=\beta_{\mathscr S}$ is the acyclic orientation of $\Cycle_n$ coming from the independent set $\mathscr S$ defined above. As before, let $\lambda\in\Lambda_{\Path_n}$, and let $\gamma$ be the size of the orbit of $\TPro_\beta$ containing $\lambda$. \Cref{lem:divisible_n-d} tells us that $\gamma$ is divisible by $n-d$, so we just need to show that $\gamma$ is also divisible by $d$. Using the division algorithm, we can write $\gamma n=q(n-d)+r$. Since $n-d$ divides $\gamma$, we find that $r=0$ and that $q$ is divisible by $n$. Thus, it follows from \Cref{prop:TProcycBro_R} that the set $J$ is empty and that we can write $\TPro_\beta^\gamma=\cyc^{-q}(\cyc\Bro_{\mathscr R})^q=(\cyc\Bro_{\mathscr R})^q$. 

Given a labeling $\sigma\in\Lambda_{\Path_n}$, let $\psi(\sigma)$ be the sequence obtained by reading the labels in $\mathscr S$ in the order in which they appear from left to right along the path in $\sigma$. Recall from \Cref{subsec:Description2} the two-step gliding-globs procedure for computing the action of $\cyc\Bro_{\mathscr R}$. In the first step of this procedure, none of the globs of liquid can glide through each other. In the second step, we simply cyclically permute the $d$ labels in the globs of liquid. This shows that $\psi(\cyc\Bro_{\mathscr R}(\sigma))$ is obtained from $\psi(\sigma)$ by cyclically permuting the labels in $\mathscr S$ in the cyclic order $s_1,\ldots,s_d$. It follows that every orbit of $\cyc\Bro_{\mathscr R}$ has size divisible by $d$. Since $\lambda=\TPro_\beta^\gamma(\lambda)=(\cyc\Bro_{\mathscr R})^q(\lambda)$, we find that $d$ divides $q$. The equation $\gamma n=q(n-d)$ then forces $d(n-d)$ to divide $\gamma n$. Since $\gcd(n,d)$ divides $n-d$, this implies that $d$ divides $\gamma(n/\gcd(n,d))$. But $d$ and $n/\gcd(n,d)$ are coprime, so $d$ divides~$\gamma$. 
\end{proof}                              

\section{Orbit Structure of Permutoric Promotion}\label{sec:Orbit_Structure}

Throughout this section, we continue to fix $G$ to be the path graph $\Path_n$. Our primary goal is to prove \Cref{thm:main}. 

\subsection{A Reformulation}\label{subsec:reformulation}
One of the advantages of \Cref{lem:counterclockwise_edges} is that it allows us to work with whichever acyclic orientation $\beta$ is most convenient for our purposes. In \Cref{sec:divisibility}, we chose to work with the acyclic orientation $\beta_{\mathscr S}$ whose sources were the elements of an independent set $\mathscr S$ and whose sinks were the elements of $\mathscr S-1$. However, in this section, we will fix $\beta$ to be the acyclic orientation of $\Cycle_n$ whose unique source is $d$ and whose unique sink is $n$. 

The purpose of \Cref{sec:divisibility} was to prove \Cref{prop:divisibility}, which tells us that the sizes of the orbits of $\TPro_\beta$ are all divisible by $\lcm(d,n-d)$. The reason this is necessary is that it allows us to reduce the problem of determining the orbit sizes of $\TPro_\beta$ to the problem of determining the orbit sizes of $\TPro_\beta^d$. The following proposition allows us to rewrite $\TPro_\beta^d$ in a more convenient form. 

\begin{proposition}\label{prop:Psi}
We have \[\TPro_\beta^d=\prod_{i=n}^1(\tau_i\tau_{i+1}\cdots\tau_{i+d-1})=(\tau_{n}\tau_{n+1}\cdots\tau_{d+n-1})\cdots(\tau_2\tau_3\cdots\tau_{d+1})(\tau_1\tau_2\cdots\tau_d).\]
\end{proposition}

\begin{proof}
Think of $(\tau_{n}\tau_{n+1}\cdots\tau_{d+n-1})\cdots(\tau_2\tau_3\cdots\tau_{d+1})(\tau_1\tau_2\cdots\tau_d)$ as a word $Y$ over the alphabet $\{\tau_1,\ldots,\tau_n\}$. Note that every letter in this alphabet appears exactly $d$ times in $Y$. By \Cref{lem:suffix}, we just need to show that if $X$ is a suffix of $Y$ and $a\to b$ is an arrow in $\beta$, then $X\!\langle a\rangle-X\!\langle b\rangle\in\{0,1\}$; this is straightforward to check directly.   
\end{proof}

\begin{remark}\label{rem:Bro_d}
By combining \Cref{prop:Psi} with the identity $\cyc^{-1}\tau_{i+1}=\tau_{i}\cyc^{-1}$ and the fact that $\cyc^n$ is the identity map, one can readily check that $\TPro_\beta^d=\left(\cyc^{-1}\Bro_{\{1,\ldots,d\}}^{-1}\right)^n$. \hfill $\triangle$
\end{remark}

Define a map $\Phi_{n,d}\colon\Lambda_{\Path_n}\to\Lambda_{\Path_n}$ by \[\Phi_{n,d}=\cyc^d\prod_{i=n-d}^1(\tau_i\tau_{i+1}\cdots\tau_{i+d-1})=\cyc^d(\tau_{n-d}\tau_{n-d+1}\cdots\tau_{n-1})\cdots(\tau_2\tau_3\cdots\tau_{d+1})(\tau_1\tau_2\cdots\tau_d).\] Using the identity $\cyc\tau_i=\tau_{i+1}\cyc$ together with \Cref{prop:Psi}, one can check that
\begin{equation}\label{eq:PhiTPro}
\Phi_{n,d}^{n/\gcd(n,d)}=\TPro_\beta^{\lcm(d,n-d)}.
\end{equation}

\begin{lemma}\label{lem:Phi_Divisible}
Every orbit of $\Phi_{n,d}\colon\Lambda_{\Path_n}\to \Lambda_{\Path_n}$ has size divisible by $n/\gcd(n,d)$. 
\end{lemma}

\begin{proof}
Let $\mathsf{FS}(\overline{\Path}_n,\Cycle_n)$ be the graph with vertex set $\Lambda_{\Path_n}$ in which two distinct labelings $\sigma,\sigma'$ are adjacent if and only if there exists $i\in\mathbb Z/n\mathbb Z$ such that $\sigma'=\tau_i(\sigma)$. In the language of the article \cite{DefantFriends}, this is the \emph{friends-and-strangers graph} of $\overline{\Path}_n$ and $\Cycle_n$, where $\overline{\Path}_n$ is the complement of $\Path_n$. For $\sigma\in\Lambda_{\Path_n}$, let $H_\sigma$ be the connected component of $\mathsf{FS}(\overline{\Path}_n,\Cycle_n)$ containing $\sigma$. It follows from Theorem~4.1 and Proposition~4.4 in \cite{DefantFriends} that there is a well-defined action of the group $\langle\cyc\rangle\cong\ZZ/n\ZZ$ on the set of connected components of $\mathsf{FS}(\overline{\Path}_n,\Cycle_n)$ given by $\cyc\cdot H_\sigma=H_{\cyc(\sigma)}$; moreover, these results from \cite{DefantFriends} imply that all orbits of this action have size $n$. Note that $H_{\Phi_{n,d}(\sigma)}=\cyc^d\cdot H_\sigma$. If $\Phi_{n,d}^k(\sigma)=\sigma$, then $H_\sigma=H_{\Phi_{n,d}^k(\sigma)}=\cyc^{dk}\cdot H_\sigma$, so $k$ is divisible by $n/\gcd(n,d)$. 
\end{proof}

Let $\Comp_d(n)$ denote the set of compositions of $n$ with $d$ parts (i.e., $d$-tuples of positive integers that sum to $n$). There is a natural \dfn{rotation} operator $\Rot_{n,d}\colon\Comp_d(n)\to\Comp_d(n)$ defined by $\Rot_{n,d}(a_1,a_2,\ldots,a_d)=(a_2,\ldots,a_d,a_1)$. 
Our goal in the next subsection will be to relate $\Phi_{n,d}$ and $\Rot_{n,d}$ via the following proposition. Recall that we write $\Orb_f$ for the set of orbits of an invertible map $f$. 

\begin{proposition}\label{prop:PhiRot}
There is a map $\Omega\colon\Orb_{\Phi_{n,d}}\to\Orb_{\Rot_{n,d}}$ such that $|\Omega(\mathcal O)|=\frac{d}{n}|\mathcal O|$ for every $\mathcal O\in\Orb_{\Phi_{n,d}}$ and $|\Omega^{-1}(\widehat{\mathcal O})|=d!(n-d)!$ for every $\widehat{\mathcal O}\in\Orb_{\Rot_{n,d}}$. 
\end{proposition}

Before proceeding to the proof of \Cref{prop:PhiRot}, let us see why it implies \Cref{thm:main}. 

\begin{proof}[Proof of \Cref{thm:main} assuming \Cref{prop:PhiRot}]
Let $k_1,\ldots,k_\ell$ be the sizes of the orbits of $\Rot_{n,d}$, and let $m_i$ be the number of orbits of $\Rot_{n,d}$ of size $k_i$. Then $\{k_i^{m_i}:1\leq i\leq \ell\}$ is the multiset of orbit sizes of $\Rot_{n,d}$, where we use superscripts to denote multiplicities. If we assume \Cref{prop:PhiRot}, then we find that the multiset of orbit sizes of $\Phi_{n,d}$ is \[\left\{\left(\frac{n}{d}k_i\right)^{d!(n-d)!m_i}:1\leq i\leq \ell\right\}.\]
It then follows from \eqref{eq:PhiTPro} and \Cref{lem:Phi_Divisible} that the multiset of orbit sizes of $\TPro_\beta^{\lcm(d,n-d)}$ is \[\left\{\left(\frac{\gcd(n,d)}{n}\frac{n}{d}k_i\right)^{(n/\gcd(n,d))d!(n-d)!m_i}:1\leq i\leq \ell\right\},\] and we can then invoke \Cref{prop:divisibility} to see that the multiset of orbit sizes of $\TPro_\beta$ is \[\left\{\left(\lcm(d,n-d)\frac{\gcd(n,d)}{n}\frac{n}{d}k_i\right)^{(1/\lcm(d,n-d))(n/\gcd(n,d))d!(n-d)!m_i}: 1\leq i\leq \ell\right\}\] \[=\left\{\left((n-d)k_i\right)^{n(d-1)!(n-d-1)!m_i}: 1\leq i\leq \ell\right\}.\] Since $\Rot_{n,d}$ has order $d$, this implies that $\TPro_\beta$ has order $d(n-d)$. It is well known \cite{CSPDefinition} that the triple \[\left(\Comp_d(n),\Rot_{n,d},{n-1 \brack d-1}_q\right)\] exhibits the cyclic sieving phenomenon. Hence, \Cref{lem:CSP_technical} (with $f=\Rot_{n,d}$ and $g=\TPro_\beta$) implies that \[\left(\Lambda_{\Path_n},\TPro_\beta,n(d-1)!(n-d-1)![n-d]_{q^d}{n-1\brack d-1}_q\right)\] exhibits the cyclic sieving phenomenon. 
\end{proof}

\subsection{Sliding Stones and Colliding Coins}
Our aim is now to prove \Cref{prop:PhiRot}, which, as we have just seen, implies our main theorem about the orbit structure of permutoric promotion. Code implementing several of the combinatorial constructions described in this section can be found at \url{https://cocalc.com/hrthomas/permutoric-promotion/implementation}.

For each integer $k$, let $\theta_k=\tau_{q+d+1-r}$, where $q$ and $r$ are the unique integers satisfying $k=qd+r$ and $1\leq r\leq d$. Let \[\nu_\ell=\theta_{d\ell}\theta_{d\ell-1}\cdots\theta_{d(\ell-1)+2}\theta_{d(\ell-1)+1}.\] Observe that $\theta_{k+dn}=\theta_k$ for all integers $k$. We have \[\Phi_{n,d}=\cyc^d\theta_{d(n-d)}\cdots\theta_2\theta_1=\cyc^d\nu_{n-d}\cdots\nu_2\nu_1.\] By combining the identity $\cyc\tau_i=\tau_{i+1}\cyc$ with the fact that $\cyc^n$ is the identity map, one can easily verify that 
\begin{equation}\label{eq:Phi_Order}
\Phi_{n,d}^{m}=\theta_{md(n-d)}\cdots\theta_2\theta_1=\nu_{m(n-d)}\cdots\nu_2\nu_1
\end{equation}
whenever $m$ is a positive multiple of $n/\gcd(n,d)$. 

Define a \dfn{state} to be a pair $(\sigma,t)\in\Lambda_{\Path_n}\times\ZZ$; we call $\sigma$ the \dfn{labeling} of the state and say that the state is at \dfn{time} $t$. A \dfn{timeline} is a bi-infinite sequence $\mathcal T=(\sigma_t,t)_{t\in\ZZ}$ of states such that $\sigma_t=\nu_t(\sigma_{t-1})$ for all $t\in\ZZ$. Note that every state belongs to a unique timeline. For $\sigma\in\Lambda_{\Path_n}$, let $\mathcal T_\sigma$ be the unique timeline containing the state $(\sigma,0)$. 

Let $v_1,\ldots,v_n$ be the vertices of $\Path_n$, listed from left to right. For each $\ell\in[n]$, let $\vv_\ell$ be a formal symbol associated to $v_\ell$; we will call $\vv_\ell$ a \dfn{replica}. Let $\s_1,\ldots,\s_d$ be stones of different colors. We define the \dfn{stones diagram} of a state $(\sigma,t)$ as follows. Start with a copy of $\Cycle_n$. Place $\s_1,\ldots,\s_d$ on the vertices $t+d,\ldots,t+1$, respectively. Then place each replica $\vv_\ell$ on the vertex $\sigma(v_\ell)$ of $\Cycle_n$; if this vertex already has a stone sitting on it, then we place the replica on top of the stone. 

Suppose we have a timeline $\mathcal T=(\sigma_t,t)_{t\in\ZZ}$. We want to describe how the stones diagrams of the states evolve as we move through the timeline. We will imagine transforming the stones diagram of $(\sigma_{t-1},t-1)$ into that of $(\sigma_t,t)$ via a sequence of $d$ \dfn{small steps}. The $i$-th small step moves $\s_i$ one space clockwise. The labeling $(\theta_{d(t-1)+i}\cdots\theta_{d(t-1)+1})(\sigma_{t-1})$ is obtained from $(\theta_{d(t-1)+i-1}\cdots\theta_{d(t-1)+1})(\sigma_{t-1})$ by applying the toggle operator $\theta_{d(t-1)+i}=\tau_{t+d-i}$. If this operator has no effect (i.e., $(\theta_{d(t-1)+i}\cdots\theta_{d(t-1)+1})(\sigma_{t-1})=(\theta_{d(t-1)+i-1}\cdots\theta_{d(t-1)+1})(\sigma_{t-1})$), then we do not move any of the replicas $\vv_1,\ldots,\vv_n$ during the $i$-th small step (in this case, the stone $\s_i$ slides from underneath one replica to underneath a different replica). Otherwise, $\theta_{d(t-1)+i}$ has the effect of swapping the labels $t+d-i$ and $t+d-i+1$, so we swap the replicas that were sitting on the vertices $t+d-i$ and $t+d-i+1$ (in this case, the stone $\s_{i}$ carries the replica sitting on it along with it as it slides). \Cref{Fig5} illustrates these small steps for a particular example with $n=8$, $d=3$, and $t=1$. 

\begin{figure}[ht]
  \begin{center}{\includegraphics[height=9.3cm]{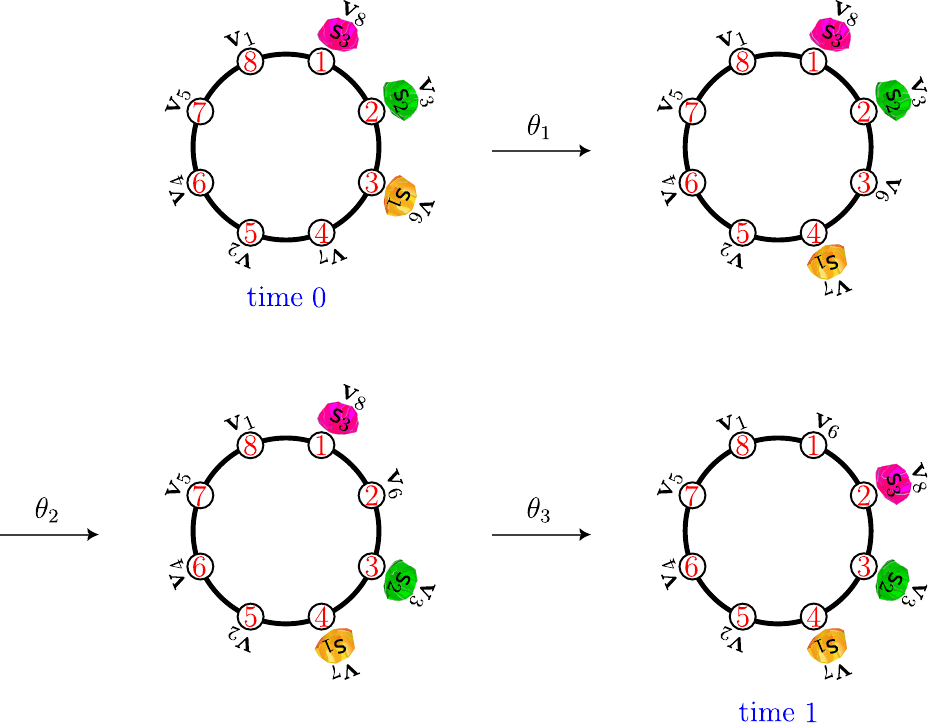}}
  \end{center}
  \caption{The $d=3$ small steps transforming the stones diagram of a state at time $0$ into the stones diagram of the next state at time $1$.}\label{Fig5}
\end{figure}

Now consider $d$ coins of different colors such that the set of colors of the coins is the same as the set of colors of the stones. We define the \dfn{coins diagram} of a state $(\sigma,t)$ as follows. Start with a copy of $\Path_n$. For each $i\in[d]$, there is a replica $\vv_\ell$ sitting on the stone $\s_i$ in the stones diagram of $(\sigma,t)$; place the coin with the same color as the stone $\s_i$ on the vertex $v_\ell$ (see \Cref{FigNewA,Fig6}). Note that the set of vertices of $\Path_n$ occupied by coins is $\{\sigma^{-1}(t+1),\ldots,\sigma^{-1}(t+d)\}$.

Consider how the coins diagrams evolve as we move through a timeline $\mathcal T=(\sigma_t,t)_{t\in\ZZ}$. Let us transform the stones diagram of $(\sigma_{t-1},t-1)$ into that of $(\sigma_t,t)$ via the $d$ small steps described above. Let $\vv_\ell$ be the replica sitting on $\s_i$ right before the $i$-th small step, and let $\vv_{\ell'}$ be the replica sitting on the vertex one step clockwise from $\s_i$ right before the $i$-th small step. When $\s_i$ moves in the $i$-th small step, it will either carry its replica $\vv_\ell$ along with it or slide from underneath $\vv_\ell$ to underneath $\vv_{\ell'}$; the latter occurs if and only if $\ell'=\ell\pm 1$. In the former case, no coins move during the $i$-th small step; in the latter case, a coin moves from $v_\ell$ to the adjacent vertex $v_{\ell'}$ (which did not have a coin on it right before this small step).  

If we watch the coins diagrams evolve as we move through the timeline, then, by the previous paragraph, the coins will move around on $\Path_n$, but they will never move through each other. Therefore, it makes sense to name the coins $\cc_1,\ldots,\cc_d$ in the order they appear along the path from left to right, and this naming only depends on the timeline (not the specific state in the timeline). Define a \dfn{traffic jam} to be a maximal nonempty collection of coins that occupy a contiguous block of vertices (so the vertices occupied by the coins in a particular traffic jam induce a connected subgraph of $\Path_n$). Note that a traffic jam could have just a single coin. We say a traffic jam \dfn{touches a wall} if it contains a coin that occupies $v_1$ or $v_n$. 

At any time, a coin has an idea of the direction in which it expects to move next (our coins are conscious now). Note that this is not necessarily the direction in which it will move next because it may change its mind before it moves. The way that a coin $\cc$ decides which direction it expects to move is as follows. Suppose $\cc$ currently occupies vertex $v_j$, and suppose the coins in the traffic jam containing $\cc$ occupy the vertices $v_r,v_{r+1},\ldots,v_s$. The coin $\cc$ looks at the stones diagram and reads ahead in the clockwise direction, starting from the stone of its color, and it determines whether it first sees $\vv_{r-1}$ or $\vv_{s+1}$. If it first sees $\vv_{r-1}$, it expects to move left; if it first sees $\vv_{s+1}$, it expects to move right. If $r-1$ is not the index of a replica (because $r=1$), the first replica that $\cc$ sees will be $\vv_{s+1}$; similarly, if $s+1$ is not the index of a replica (because $s=n$), the first replica $\cc$ sees will be $\vv_{r-1}$. 

\Cref{FigNewA} shows several stones diagrams and coins diagrams. In each coins diagram, an arrow has been placed over each coin to indicate which direction it expects to move. 

\begin{lemma}\label{lem:move_expected} 
When a coin moves, it moves in the direction that it expects to move. \end{lemma}

\begin{proof} 
Suppose $\cc$ occupies $v_j$ and is about to move left. The stone of the same color as $\cc$ is under the
replica $\vv_j$, and the next replica clockwise is $\vv_{j-1}$, which has no stone under it. It follows that $\cc$ is the leftmost coin in a traffic jam. When $\cc$ reads through the stones diagram looking for one of two replicas, it first sees $\vv_{j-1}$, so it indeed expects to move left. The analysis for coins moving to the right is the same. 
\end{proof}

The following lemma tells us under what circumstances a coin can change its mind about which way it is going to move. 

\begin{lemma} 
Let $\cc$ be a coin, and let $\s$ be the stone with the same color as $\cc$. Consider a small step, and let $v_j$ be the vertex occupied by $\cc$ right before the small step. Let $v_{r},v_{r+1},\ldots,v_s$ be the vertices occupied by the coins in the traffic jam that contains $\cc$ right before the small step. During this small step, $\cc$ changes its mind about which direction it expects to move if and only if one of the following situations occurs: 
\begin{itemize}
\item The stone $\s$ slides through $\vv_{r-1}$ or $\vv_{s+1}$ and carries $\vv_j$ along with it as it slides (so $\cc$ does not move in the coins diagram), and the traffic jam containing $\cc$ does not touch a wall (so $1<r\leq s<n$). 
\item The coin $\cc$ moves, and the traffic jam that contains $\cc$ after the small step touches a wall. 
\end{itemize}
\end{lemma} 

\begin{proof}
First of all, note that $\cc$ will not change its mind about which way it is going to move except during a small step when $\s$ moves. Indeed, even though the traffic jam containing $\cc$ may change during other small steps, it is straightforward to check that these small steps will not change the direction that $\cc$ expects to move.

While $\cc$ is in a traffic jam that touches a wall, there is only one way that it can expect to move: away from that wall. Thus, it does not change its mind about which way it is moving before it actually moves, but it does change its mind the moment it arrives in the traffic jam (i.e., when the second bulleted item in the statement of the lemma is satisfied).

Now consider a small step during which $\cc$ moves, and suppose the traffic jam that contains $\cc$ after the small step does not touch a wall. For simplicity, let us assume that $\cc$ expects to move left before this small step. Then during the small step, $\cc$ does in fact move left (by \cref{lem:move_expected}). Let us say $\cc$ moves from $v_j$ to $v_{j-1}$. Then during this small step, $\s$ slides from underneath the replica $\vv_j$ to underneath the replica $\vv_{j-1}$. After the small step, the vertices occupied by the coins in the traffic jam containing $\cc$ are $v_r,v_{r+1},\ldots,v_{j-1}$ for some $r\in\{2,\ldots,j-1\}$, so $\cc$ looks in the stones diagram for either $\vv_{r-1}$ or $\vv_j$. It will certainly see $\vv_{r-1}$ first since $\vv_j$ is one step behind $\s$ (in the clockwise order) at this time. Thus, $\cc$ still expects to move left after the small step.  

Finally, consider the situation from the first bulleted item in the statement of the lemma. Let us again assume for simplicity that $\cc$ expects to move left before the small step. Then $\s$ slides through $\vv_{r-1}$. After the small step, when $\cc$ reads through the stones diagram to determine which direction it expects to move, it again searches for $\vv_{r-1}$ and $\vv_{s+1}$ (because no coins moved during the small step). It will see $\vv_{s+1}$ before $\vv_{r-1}$ because $\vv_{r-1}$ is now right behind $\s$ in the clockwise order. So $\cc$ expects to move right after the small step. 
\end{proof} 

The importance of understanding the direction in which a coin expects to move is that it will enable us to understand \dfn{collisions}. 
There are \dfn{two-coins collisions}, which involve two coins that occupy adjacent vertices of $\Path_n$; there are \dfn{left-wall collisions}, which can occur when $\cc_1$ occupies $v_1$; and there are \dfn{right-wall collisions}, which can occur when $\cc_d$ occupies $v_n$. 
The prototypical examples of collisions are when two non-adjacent coins move to become adjacent or when a coin moves to become adjacent to a wall, but other examples are possible when traffic jams of size greater than one are involved.

The precise definition of a two-coins collision that occurs in a traffic jam that does not touch a wall is as follows. We say coins $\cc_i$ and $\cc_{i+1}$ are \dfn{butting heads} if they occupy adjacent vertices and $\cc_i$ expects to move right while $\cc_{i+1}$ expects to move left. We say $\cc_{i}$ and $\cc_{i+1}$ are involved in a two-coins collision at a small step if they are not butting heads immediately before the small step and they are butting heads immediately after the small step. 
This can happen either because the two coins were not adjacent prior to the small step, but it can also happen because the two coins were adjacent but one of them changed its mind about the direction it expected to move. 

The definition has to be slightly modified in a traffic jam that touches a wall. Consider first the case when a small step occurs during which a coin $\cc$ moves so as to join a traffic jam that touches the wall. At the same time, $\cc$ changes its mind so that it now expects to move away from the wall that the traffic jam touches. Nonetheless, if there is a coin $\cc'$ adjacent to $\cc$ after the small step, we still count this as a two-coins collision between $\cc$ and $\cc'$. (We can imagine that there was a brief instant of time right after $\cc$ moved to join the traffic jam but right before it changed its mind about which way it expected to move, thus resulting in $\cc$ butting heads with $\cc'$ very briefly.)
Similarly, if $\cc$ moved onto $v_1$ (respectively, $v_n$) during this small step (so it is in a traffic jam of size $1$ that touches a wall), then we count this as a left-wall (respectively, right-wall) collision. 

We now discuss how to define a collision that occurs in the ``interior'' of a traffic jam of size at least $2$ that touches a wall. In such a traffic jam, all the coins always want to move away from the wall, so by the above definition, there would be no collisions within the traffic jam. However, this is not what we want. Instead, suppose we are considering a coin $\cc_i$ that occupies $v_j$. Assume the coins in the traffic jam containing $\cc_i$ occupy vertices $v_{1},v_2,\ldots,v_k$, where $j<k$. Thus, we are assuming the traffic jam touches the left wall, but the symmetrical considerations apply if the traffic jam touches the right wall.
The stone with the same color as $\cc_i$ carries the replica
$\vv_{j}$. Suppose there is a small step during which the stone with the same color as $\cc_i$ slides through $\vv_{k+1}$, carrying $\vv_j$ along with it as it slides. Note that $\cc_i$ does not move during this small step. In this case, we say $\cc_i$ collides with $\cc_{i-1}$ (or is involved in a left-wall collision if $i=1$). To explain heuristically why this collision occurs, we can imagine that $\cc_i$ has a ``flicker of confusion'' when it sees the stone with its same color slide through $\vv_{k+1}$. When it sees this, $\cc_i$ ``thinks'' it should change its mind and expect to move left. But then it realizes that it cannot expect to move left because it is in a traffic jam that touches the left wall, so it quickly goes back to expecting to move right. During this brief instant, the collision occurs because $\cc_i$ ``thinks'' it should be butting heads with $\cc_{i-1}$ (or with the left wall if $i=1$).  

We say a collision occurs at time $t$ if it occurs during a small step between times $t-1$ and $t$.

\begin{example}
Suppose $n=6$ and $d=3$. \Cref{FigNewA} shows some stones diagrams and coins diagrams evolving over time. At each stage, the arrow over a coin points in the direction that the coin expects to move. Collisions are indicated in the coins diagrams by stars, and each star is colored to indicate which stone moves in the small step during which the collision occurs. Note that the right-wall collision at time $5$ (marked with a gold star in the first small step after time $4$) occurs because $\cc_3$ has a ``flicker of confusion'' when the gold stone $\s_1$ slides through $\vv_4$ (carrying $\vv_6$ along with it as it slides). \hfill $\lozenge$
\end{example}
\begin{figure}
  \begin{center}{\includegraphics[width=\linewidth]{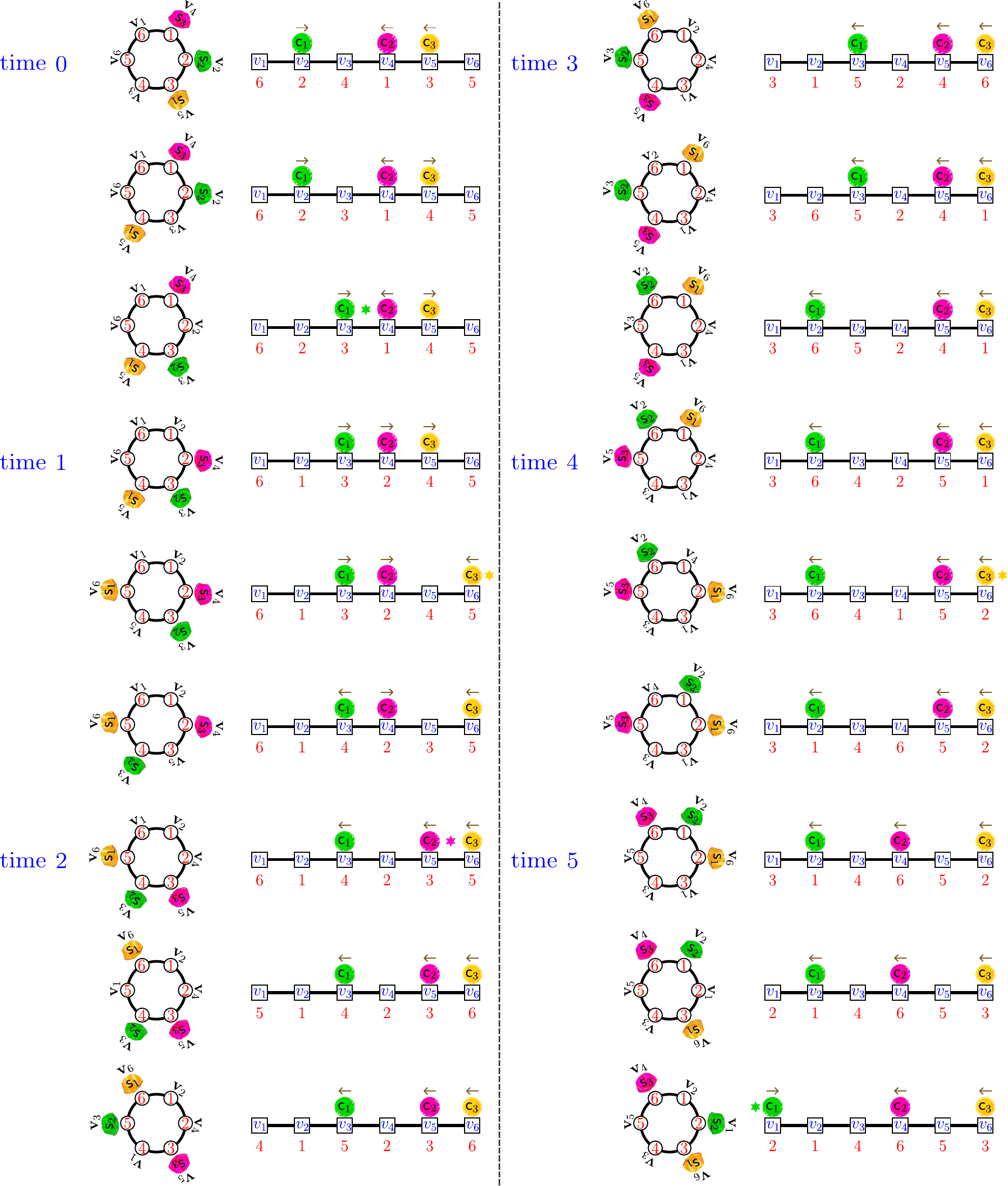}}
  \end{center}
\caption{The evolution of stones diagrams and coins diagrams over time, with each individual small step illustrated. At each moment, we have drawn an arrow over each coin to indicate which direction it expects to move. Each collision is indicated by a star whose color is the same as that of the stone that moved to cause the collision. Each labeling is depicted in red numbers below the path.}\label{FigNewA}
\end{figure}

\begin{example}\label{exam:3}
Suppose $n=6$ and $d=3$. \Cref{Fig6} shows the stones diagrams and coins diagrams of a particular timeline at times $0,1,\ldots,17$. For brevity, we have not shown the individual small steps. All of the collisions the occur at time $t$ (i.e., during the small steps between time $t-1$ and time $t$) are indicated in the coins diagram at time $t$. The color of the star can be used to determine the small step during which the collision occurs. One can check that the states in this timeline are periodic with period $18$. 
\hfill $\lozenge$
\end{example}

\begin{figure}
  \begin{center}{\includegraphics[width=\linewidth]{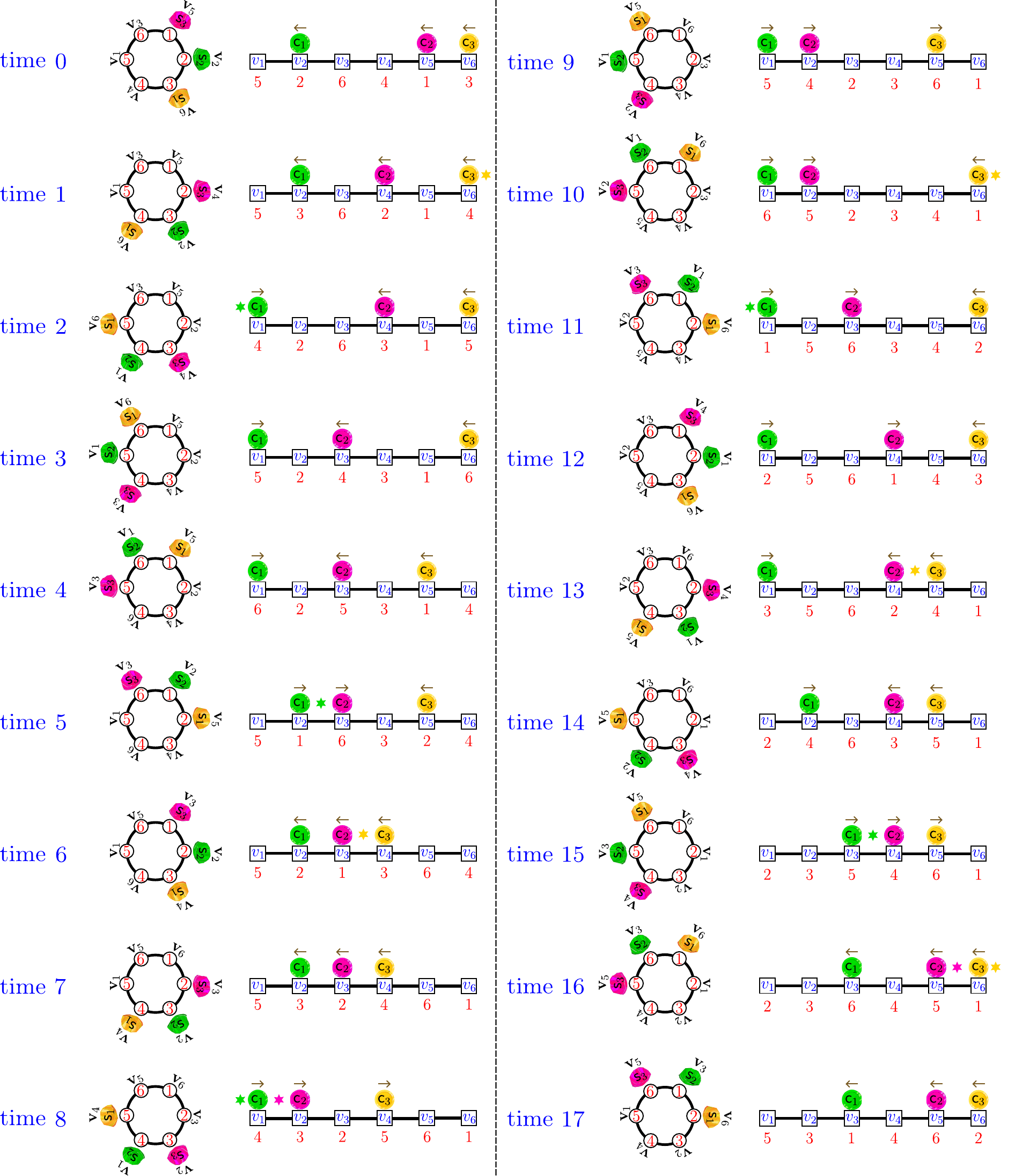}}
  \end{center}
\caption{The stones diagrams and coins diagrams of the states in a timeline at times $0,1,\ldots,17$. Here, $n=6$ and $d=3$. The collisions that occur during the small steps between times $t-1$ and $t$ are represented by color-coded stars in the coins diagram at time $t$. Each labeling is depicted by the red numbers below the path. }\label{Fig6}
\end{figure}
 
Let $\Coll_{\mathcal T}$ be the set of all collisions that take place in the coins diagrams of the states of the timeline $\mathcal T$. We define a directed graph with vertex set $\Coll_{\mathcal T}$ by drawing an arrow from a collision $\kappa$ to a collision $\kappa'$ whenever there is a coin involved in both $\kappa$ and $\kappa'$ and the collision $\kappa$ occurs before $\kappa'$. Let $(\Coll_{\mathcal T},\leq_{\mathcal T})$ be the transitive closure of this directed graph. Let ${\bf H}_{\mathcal T}$ be the Hasse diagram of $(\Coll_{\mathcal T},\leq_{\mathcal T})$. This Hasse diagram, which will be one of our primary tools, has the shape of a bi-infinite chain link fence (see \Cref{Fig7}). Suppose $\kappa_1\lessdot_{\mathcal T}\kappa_2$ is an edge in ${\bf H}_{\mathcal T}$. Then $\kappa_1$ and $\kappa_2$ are collisions that both use some coin $\cc$; we define the \dfn{energy} of this edge, denoted $\mathcal E(\kappa_1\lessdot_{\mathcal T}\kappa_2)$, to be the number of different vertices that $\cc$ occupies between these two collisions, including the vertices occupied by $\cc$ when the collisions occur. More generally, if $\kappa_1\lessdot_{\mathcal T}\kappa_2\lessdot_{\mathcal T}\cdots\lessdot_{\mathcal T}\kappa_r$ is a saturated chain in ${\bf H}_{\mathcal T}$, then we write $\mathcal E(\kappa_1\lessdot_{\mathcal T}\kappa_2\lessdot_{\mathcal T}\cdots\lessdot_{\mathcal T}\kappa_r)$ for the tuple $(\mathcal E(\kappa_1\lessdot_{\mathcal T}\kappa_2),\ldots,\mathcal E(\kappa_{r-1}\lessdot_{\mathcal T}\kappa_r))$ of energies of the edges in the chain.

\begin{example}\label{exam:4}
If $\mathcal T$ is the timeline containing the states whose stones diagrams and coins diagrams are shown in \Cref{Fig6}, then (a finite part of) ${\bf H}_{\mathcal T}$ is shown in \Cref{Fig7}. Each collision is represented by a color-coded star, and the blue number inside the star is the time when the collision occurs. Each edge is labeled by its energy.  
\hfill $\lozenge$ 
\end{example}

\begin{figure}[ht]
  \begin{center}{\includegraphics[height=11.322cm]{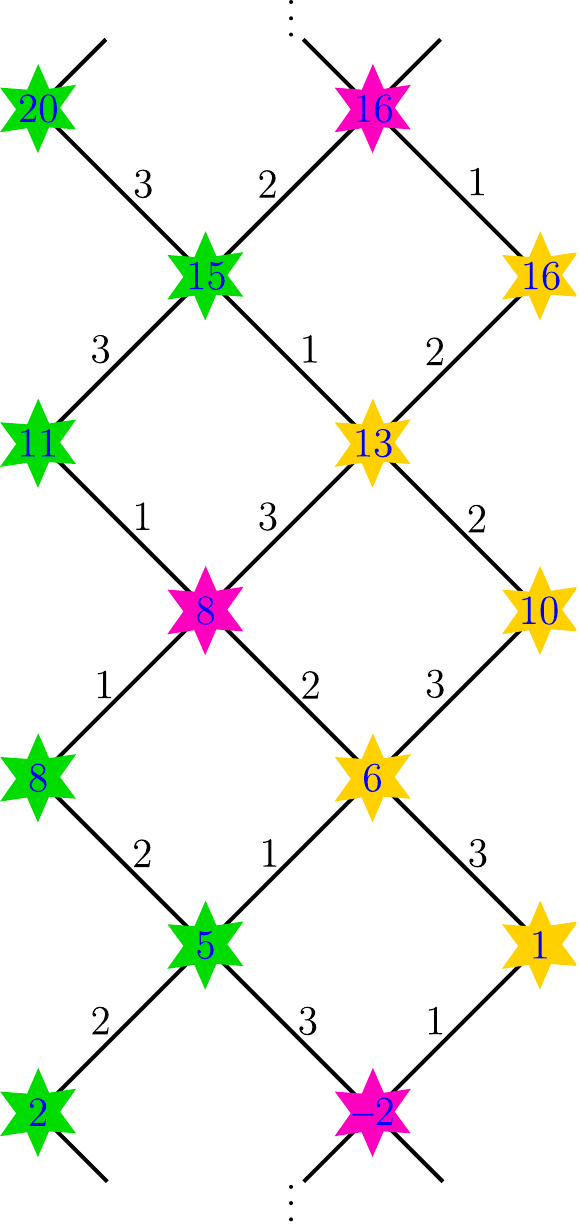}}
  \end{center}
\caption{A Hasse diagram ${\bf H}_{\mathcal T}$. Each collision is represented by a star whose color is the same as that of the stone that moved to cause the collision. Blue numbers indicate the times when the collisions occur. Edges are labeled by their energies. }\label{Fig7}
\end{figure}

A \dfn{diamond} in ${\bf H}_{\mathcal T}$ consists of collisions $\kappa_1,\kappa_2,\kappa_3,\kappa_4$ together with four edges given by cover relations $\kappa_1\lessdot_{\mathcal T}\kappa_2$, $\kappa_1\lessdot_{\mathcal T}\kappa_3$, $\kappa_2\lessdot_{\mathcal T}\kappa_4$, $\kappa_3\lessdot_{\mathcal T}\kappa_4$. 
A \dfn{half-diamond} in ${\bf H}_{\mathcal T}$ consists of collisions $\kappa_1',\kappa_2',\kappa_3'$, where $\kappa_1'$ and $\kappa_3'$ are either both left-wall collisions or both right-wall collisions, together with two edges given by cover relations $\kappa_1'\lessdot_{\mathcal T}\kappa_2'$ and $\kappa_2'\lessdot_{\mathcal T}\kappa_3'$. 
 Our arguments in the next subsection rest on the following three lemmas.

\begin{lemma}\label{lem:half-diamond}
In any half-diamond in the Hasse diagram $\bf H_{\mathcal T}$, the two edges have the same energy. 
\end{lemma}

\begin{proof}
Fix a half-diamond in ${\bf H}_{\mathcal T}$ with edges given by the cover relations $\kappa_1'\lessdot_{\mathcal T}\kappa_2'\lessdot_{\mathcal T}\kappa_3'$. By symmetry, we may assume $\kappa_1'$ and $\kappa_3'$ are both left-wall collisions. Then $\kappa_1'$ and $\kappa_3'$ occur when $\cc_1$ occupies $v_1$. Say $\kappa_2'$ occurs when $\cc_1$ occupies $v_m$. Both edges of the half-diamond have energy $m$.  
\end{proof}

\begin{lemma}\label{lem:diamond}
In any diamond in the Hasse diagram $\bf H_{\mathcal T}$, opposite edges have the same energy. 
\end{lemma}

\begin{proof}
Fix a diamond in ${\bf H}_{\mathcal T}$ with edges given by cover relations $\kappa_1\lessdot_{\mathcal T}\kappa_2$, $\kappa_1\lessdot_{\mathcal T}\kappa_3$, $\kappa_2\lessdot_{\mathcal T}\kappa_4$, $\kappa_3\lessdot_{\mathcal T}\kappa_4$. Say the collision $\kappa_1$ involves coins $\cc_i$ and $\cc_{i+1}$ and takes place when $\cc_i$ occupies vertex $v_m$ and $\cc_{i+1}$ occupies vertex $v_{m+1}$. Without loss of generality, suppose $\kappa_2$ involves $\cc_i$ and $\kappa_3$ involves $\cc_{i+1}$. Then $\kappa_2$ occurs when $\cc_i$ occupies some vertex $v_k$ with $k\leq m$, and $\kappa_3$ occurs when $\cc_{i+1}$ occupies some vertex $v_\ell$ with $\ell\geq m+1$. The collision $\kappa_4$ involves the coins $\cc_i$ and $\cc_{i+1}$. We claim that $\kappa_4$ occurs when $\cc_i$ occupies $v_{k+\ell-m-1}$ and $\cc_{i+1}$ occupies $v_{k+\ell-m}$; this will imply that the edges $\kappa_1\lessdot_{\mathcal T}\kappa_2$ and $\kappa_3\lessdot_{\mathcal T} \kappa_4$ both have energy $m-k+1$ and that the edges $\kappa_1\lessdot_{\mathcal T}\kappa_3$ and $\kappa_2\lessdot_{\mathcal T} \kappa_4$ both have energy $\ell-m$.

By symmetry, we may assume that $m-k\leq \ell-m-1$. Let $x,y\in[d]$ be such that $\s_x$ and $\s_y$ are the stones with the same colors as $\cc_i$ and $\cc_{i+1}$, respectively. 
Consider starting at the time when $\kappa_1$ occurs and watching the stones diagrams and coins diagrams evolve as we move forward in time. We will assume that $k<m$, that $x>y$, and that $\cc_i$ moves from $v_m$ to $v_{m-1}$ before $\cc_{i+1}$ moves from $v_{m+1}$ to $v_{m+2}$; the other cases are similar. Let $t_0$ be the first time after the collision $\kappa_1$ when $\cc_i$ moves from $v_m$ to $v_{m-1}$. The coin $\cc_i$ will move from $v_m$ to $v_{m-1}$ and then to $v_{m-2}$ and so on until reaching $v_k$; it will then turn around and move back across $v_{k+1},\ldots,v_m$ and then continue on toward $v_{k+\ell-m-1}$. For $j\in\{k,\ldots,k+\ell-m-1\}$, let $\zeta_j$ be the amount of time that $\cc_i$ spends on $v_j$ during this trip. The coin $\cc_{i+1}$ stays on $v_{m+1}$ for some time after $t_0$; it then moves to the right until reaching $v_{\ell}$, where it turns around and heads back to $v_{k+\ell-m+1}$. For $j'\in\{m+1,\ldots,\ell\}$, let $\xi_{j'}$ be the amount of time after $t_0$ that $\cc_{i+1}$ spends on $v_{j'}$ during this trip.

By analyzing the stones diagrams, one can show that $\zeta_j=\xi_{j'}=n-d$ for all $j\in\{k,\ldots,m-1\}$ and all $j'\in\{k+\ell-m+1,\ldots,\ell\}$. Similarly, $\zeta_j=\xi_{j+1}$ for all $j\in\{m,\ldots,k+\ell-m-1\}$. Let $N=(m-k)(n-d)+\sum_{j=m}^{k+\ell-m-1}\zeta_j$. At time $t_0+N$, either the coin $\cc_i$ moves from $v_{k+\ell-m-1}$ to $v_{k+\ell-m}$, or the coin $\cc_{i+1}$ moves from $v_{k+\ell-m+1}$ to $v_{k+\ell-m}$. It follows from the assumption that $x>y$ that, in fact, $\cc_{i+1}$ moves from $v_{k+\ell-m+1}$ to $v_{k+\ell-m}$ at time $t_0+N$. This proves the claim.  
\end{proof}

\begin{example}\label{exam:2}
Suppose $n=6$ and $d=3$, and let $\mathcal T$ be the timeline from \Cref{exam:3,exam:4}. Let $\kappa_1,\kappa_3,\kappa_4$ be the collisions that occur at times $6,10,13$, respectively, and let $\kappa_2$ be the two-coins collision at time $8$. In the notation of the proof of \Cref{lem:diamond}, we have $i=2$, $m=3$, $k=2$, $\ell=6$, and $t_0=8$. We have $\zeta_2=\xi_6=3=n-d$, $\zeta_3=\xi_4=1$, and $\zeta_4=\xi_5=1$. Thus, $N=(m-k)(n-d)+\sum_{j=m}^{k+\ell-m-1}\zeta_j=5$. As explained in the proof of \Cref{lem:diamond}, the coin $\cc_3$ moves from $v_6$ to $v_5$ at time $t_0+N=13$. 
\hfill $\lozenge$
\end{example}

\begin{lemma} \label{lem:new}
If the two edges of a half-diamond in the Hasse diagram $\bf H_{\mathcal T}$ have energy $m$, then the amount of time between the collisions at the bottom and the top of the half-diamond is $m(n-d)$.
\end{lemma}

\begin{proof}
Without loss of generality, assume the bottom and top collisions in the half-diamond are left-wall collisions. 

If $m>1$, the same style of argument as in the proof of \cref{lem:diamond} proves that the total amount of time that $\cc_1$ spends on each of the vertices $v_1,\ldots,v_m$ is 
exactly $n-d$, which proves the claim.

If $m=1$, a similar argument also applies. 
Let $\s$ be the stone with the same color as $\cc_1$. A left-wall collision can only occur during a small step in which $\s$ moves. Moreover, such a small step results in a left-wall collision if and only if, right before the small step occurs, the replica one space clockwise of $\s$ is $\vv_{\ell}$, where $\ell$ is the smallest of all the indices of replicas that do not sit on stones at that time (equivalently, the traffic jam containing $\cc_1$ has size $\ell-1$). The relative cyclic order of the indices of the replicas that do not sit on stones remains constant over time, so the time between these left-wall collisions is exactly $n-d$. 
\end{proof}

\subsection{The Map $\Omega$}
Equipped with \cref{lem:half-diamond,lem:diamond,lem:new}, we now turn to constructing and analyzing the map $\Omega$ from \Cref{prop:PhiRot}.

For each collision $\kappa\in\Coll_{\mathcal T}$, let $\varphi(\kappa)$ be the collision involving the same set of coins as $\kappa$ that occurs next after $\kappa$. In other words, if $\kappa$ is the bottom element of a diamond (respectively, half-diamond), then $\varphi(\kappa)$ is the top element of that same diamond (respectively, half-diamond). We extend this notation to saturated chains in ${\bf H}_{\mathcal T}$ (including edges) by letting \[\varphi(\kappa_1\lessdot_{\mathcal T}\kappa_2\lessdot_{\mathcal T}\cdots\lessdot_{\mathcal T}\kappa_m)=\varphi(\kappa_1)\lessdot_{\mathcal T}\varphi(\kappa_2)\lessdot_{\mathcal T}\cdots\lessdot_{\mathcal T}\varphi(\kappa_m).\] We define the \dfn{period} of ${\bf H}_{\mathcal T}$ to be the smallest positive integer $p$ such that $e$ and $\varphi^p(e)$ have the same energy for every edge $e$ of ${\bf H}_{\mathcal T}$. A \dfn{transversal} of ${\bf H}_{\mathcal T}$ is a saturated chain $\mathscr T=(\kappa_0\lessdot_{\mathcal T}\kappa_1\lessdot_{\mathcal T}\cdots\lessdot_{\mathcal T}\kappa_d)$ such that $\kappa_0$ is a left-wall collision, $\kappa_d$ is a right-wall collision, and $\kappa_i$ involves the stones $\cc_i$ and $\cc_{i+1}$ for every $i\in[d-1]$. In other words, a transversal is a saturated chain that moves from left to right across ${\bf H}_{\mathcal T}$. We define the \dfn{energy composition} of $\mathscr T$ to be the tuple $\mathcal{E}(\mathscr T)=(\varepsilon_1,\ldots,\varepsilon_d)$, where $\varepsilon_i$ is the energy of the edge $\kappa_{i-1}\lessdot_{\mathcal T}\kappa_i$; note that $\mathcal{E}(\mathscr T)\in\Comp_d(n)$.

\begin{lemma}\label{lem:ERot}
Let $\mathcal T$ be a timeline, and let $\mathscr T$ be a transversal of ${\bf H}_{\mathcal T}$. Then \[\mathcal{E}(\varphi(\mathscr T))=\Rot_{n,d}(\mathcal{E}(\mathscr T)).\] Moreover, the period of ${\bf H}_{\mathcal T}$ is equal to the size of the orbit of $\Rot_{n,d}$ containing $\mathcal{E}(\mathscr T)$. 
\end{lemma}

\begin{proof}
The second statement follows from the first because, by \Cref{lem:half-diamond,lem:diamond}, the energies of all edges in ${\bf H}_{\mathcal T}$ are determined by the energy composition of a single transversal of ${\bf H}_{\mathcal T}$. The first statement is also immediate from \Cref{lem:half-diamond,lem:diamond}. 
\end{proof}

\begin{example}
Suppose $n=6$ and $d=3$. Let ${\bf H}_{\mathcal T}$ be the Hasse diagram from \Cref{Fig7}, and let $\mathscr T=(\kappa_0\lessdot_{\mathcal T}\kappa_1\lessdot_{\mathcal T}\kappa_2\lessdot_{\mathcal T}\kappa_3)$ be the transversal consisting of the collisions that occur at times $2,5,6,10$. Then $\mathcal{E}(\mathscr T)=(2,1,3)\in\Comp_3(6)$. The period of ${\bf H}_{\mathcal T}$ is $3$, which is the size of the $\Rot_{6,3}$-orbit containing $(2,1,3)$. The transversal $\varphi(\mathscr T)$ consists of both the collisions that occur at time $8$ along with the collisions at times $13$ and $16$. We have $\mathcal{E}(\varphi(\mathscr T))=(1,3,2)=\Rot_{6,3}(\mathcal{E}(\mathscr T))$. Similarly, $\mathcal{E}(\varphi^2(\mathscr T))=(3,2,1)=\Rot_{6,3}^2(\mathcal{E}(\mathscr T))$. \hfill $\lozenge$ 
\end{example}

Let $S_r$ be the symmetric group consisting of all permutations of $[r]$. Suppose $v_{i_1},\ldots,v_{i_r}$ is a sequence of distinct vertices of $\Path_n$. We define the \dfn{standardization} of this sequence to be the unique permutation in $S_r$ that has the same relative order as $i_1,\ldots,i_r$ when written in one-line notation. For example, the standardization of $v_3,v_5,v_1,v_6$ is $2314$. Let $\mathcal T=(\sigma_t,t)_{t\in\ZZ}$ be a timeline. Recall that the stones $\s_1,\ldots,\s_d$ sit on the vertices $t+d,\ldots,t+1$, respectively, in the stones diagram of $(\sigma_t,t)$. Let $\stand_t(\mathcal T)$ be the standardization of the sequence $\sigma_{t}^{-1}(t+d),\ldots,\sigma_{t}^{-1}(t+1)$. Alternatively, $\stand_t(\mathcal T)$ is the permutation $\rho\colon[d]\to[d]$ such that the stone $\s_i$ and the coin $\cc_{\rho(i)}$ have the same color for every $i\in[d]$. Let us also define $\overline\stand_t(\mathcal T)$ to be the standardization of $\sigma_{t}^{-1}(1),\ldots,\sigma_{t}^{-1}(t),\sigma_{t}^{-1}(t+d+1),\ldots,\sigma_{t}^{-1}(n)$ (i.e., the standardization of the sequence obtained from $\sigma_{t}^{-1}(1),\sigma_{t}^{-1}(2),\ldots,\sigma_{t}^{-1}(n)$ by deleting $\sigma_{t}^{-1}(i)$ for all $t+1\leq i\leq t+d$). It follows from the analysis of how stones diagrams evolve through a timeline that $\stand_t(\mathcal T)=\stand_{t+1}(\mathcal T)$ and $\overline\stand_t(\mathcal T)=\overline\stand_{t+1}(\mathcal T)$. In other words, $\stand_t(\mathcal T)$ and $\overline\stand(\mathcal T)$ only depend on the timeline $\mathcal T$ and not on the time $t$. Thus, it makes sense to drop the subscripts and just write $\stand(\mathcal T)$ and $\overline\stand(\mathcal T)$. Note that there are $d!$ possibilities for $\stand(\mathcal T)$ and $(n-d)!$ possibilities for $\overline\stand(\mathcal T)$; this will end up being responsible for the appearance of $d!(n-d)!$ in \Cref{prop:PhiRot}. 

For $k,t\in\ZZ$, let $\sigma_t^{(k)}=\cyc^{-k}(\sigma_{t+k})$. It follows immediately from the definition of a timeline that the sequence $\mathcal T^{(k)}=(\sigma_t^{(k)},t)_{t\in\ZZ}$ is also a timeline; that is, $\nu_t(\sigma_{t-1}^{(k)})=\sigma_t^{(k)}$ for all $t\in\ZZ$. Furthermore, the stones diagram of $(\sigma_t^{(k)},t)$ is obtained from that of $(\sigma_{t+k},t+k)$ by moving all stones and replicas $k$ positions counterclockwise. It follows that the coins diagrams of $(\sigma_t^{(k)},t)$ and $(\sigma_{t+k},t+k)$ are identical. Therefore, if $\kappa$ is a collision in $\Coll_{\mathcal T^{(k)}}$ that occurs at time $t$, then there is a collision $\psi_k(\kappa)\in\Coll_{\mathcal T}$ that occurs at time $t+k$. The resulting map $\psi_k\colon\Coll_{\mathcal T^{(k)}}\to\Coll_{\mathcal T}$ is an isomorphism from $(\Coll_{\mathcal T^{(k)}},\leq_{\mathcal T^{(k)}})$ to $(\Coll_{\mathcal T},\leq_{\mathcal T})$; furthermore, under this isomorphism, corresponding edges of the Hasse diagrams ${\bf H}_{\mathcal T^{(k)}}$ and ${\bf H}_{\mathcal T}$ have the same energy. 

Recall that we write $\mathcal T_\sigma$ for the unique timeline containing the state $(\sigma,0)$. It follows from \Cref{lem:ERot} that the energy compositions of the transversals of ${\bf H}_{\mathcal T_\sigma}$ form a single orbit $\widetilde\Omega(\sigma)$ of $\Rot_{n,d}$. If $\mathcal T_\sigma=(\sigma_t,t)_{t\in\ZZ}$ (so $\sigma_0=\sigma$), then $\Phi_{n,d}(\sigma_0)=\sigma_0^{(n-d)}$, so $\mathcal T_{\Phi_{n,d}(\sigma_0)}=\mathcal T_\sigma^{(n-d)}$. Using the isomorphism $\psi_{n-d}$, we find that $\widetilde\Omega(\sigma_0)=\widetilde\Omega(\Phi_{n,d}(\sigma_0))$. Thus, we obtain a map \[\Omega=\Omega_{n,d}\colon\Orb_{\Phi_{n,d}}\to\Orb_{\Rot_{n,d}}\] that sends the $\Phi_{n,d}$-orbit containing a labeling $\mu$ to $\widetilde\Omega(\mu)$. We will prove that this map satisfies the conditions in \Cref{prop:PhiRot}. 

\begin{lemma}\label{lem:depends_orbit}
For any labeling $\sigma\in\Lambda_{\Path_n}$, we have \[\stand(\mathcal T_\sigma)=\stand(\mathcal T_{\Phi_{n,d}(\sigma)})\quad\text{and}\quad\overline\stand(\mathcal T_\sigma)=\overline\stand(\mathcal T_{\Phi_{n,d}(\sigma)}).\] Hence, $\stand(\mathcal T_\sigma)$ and $\overline\stand(\mathcal T_\sigma)$ only depend on the orbit of $\Phi_{n,d}$ containing $\sigma$. 
\end{lemma}

\begin{proof}
Let $\mathcal T_\sigma=(\sigma_t,t)_{t\in\ZZ}$ (so $\sigma_0=\sigma$). Let $\mu=\Phi_{n,d}(\sigma)$. The stones diagram of the state $(\sigma_0^{(n-d)},0)=(\mu,0)$ is obtained from that of $(\sigma_{n-d},n-d)$ by moving all stones and replicas $n-d$ positions counterclockwise, so $\stand(\mathcal T_\sigma)=\stand(\mathcal T_{\Phi_{n,d}(\sigma)})$. Since $(\sigma_{n-d},n-d)$ is in the timeline $\mathcal T_\sigma$ and $(\mu,0)$ is in the timeline $\mathcal T_\mu$, the permutations $\overline\stand(\mathcal T_\sigma)$ and $\overline\stand(\mathcal T_\mu)$ are the standardizations of the sequences $\sigma_{n-d}^{-1}(1),\sigma_{n-d}^{-1}(2),\ldots,\sigma_{n-d}^{-1}(n-d)$ and $\mu^{-1}(d+1),\mu^{-1}(d+2),\ldots,\mu^{-1}(n)$, respectively. But $\mu=\cyc^d(\sigma_{n-d})$, so these sequences are equal. 
\end{proof}

For $\rho\in S_r$, let $\rev(\rho)$ be the permutation whose one-line notation is obtained by reversing that of $\rho$. Let $\delta\colon\ZZ/n\ZZ\to\ZZ/n\ZZ$ be the automorphism of $\Cycle_n$ defined by $\delta(i)=d+1-i$. Given an orbit $\widehat{\mathcal O}\in\Orb_{\Rot_{n,d}}$, let $\rev(\widehat{\mathcal O})\in\Orb_{\Rot_{n,d}}$ be the orbit obtained by reversing all the compositions in $\widehat{\mathcal O}$. 

\begin{lemma}\label{lem:reverse}
For every $\sigma\in\Lambda_{\Path_n}$, we have \[\stand(\mathcal T_{\delta\circ\sigma})=\rev(\stand(\mathcal T_\sigma))\quad\text{and}\quad\overline\stand(\mathcal T_{\delta\circ\sigma})=\rev(\overline\stand(\mathcal T_\sigma)).\] Furthermore, $\widetilde\Omega(\delta\circ\sigma)=\rev(\widetilde\Omega(\sigma))$. 
\end{lemma}

\begin{proof}
The first statement is immediate from the definitions. To see why the second statement is true, note that we can obtain the coins diagrams of the states in $\mathcal T_{\delta\circ\sigma}$ by ``going backward in time'' through the coins diagrams of the states in $\mathcal T_\sigma$ and permuting the colors of the coins. To be more precise, let us write $\mathcal T_\sigma=(\sigma_t,t)_{t\in\ZZ}$ and $\mathcal T_{\delta\circ\sigma}=(\sigma_t',t)_{t\in\ZZ}$ (so $\sigma_0=\sigma$ and $\sigma_0'=\delta\circ\sigma$). Then for every $t\in\ZZ$, the coins diagram of $(\sigma_t',t)$ is obtained from that of $(\sigma_{-t},-t)$ by permuting the colors of the coins. Let $\mathscr T=(\kappa_0\lessdot_{\mathcal T_\sigma}\cdots\lessdot_{\mathcal T_\sigma}\kappa_d)$ be a transversal of ${\bf H}_{\mathcal T_\sigma}$ with energy composition $\mathcal E(\mathscr T)=(\varepsilon_1,\ldots,\varepsilon_d)$. Then $\widetilde\Omega(\sigma)$ is the orbit of $\Rot_{n,d}$ containing $(\varepsilon_1,\ldots,\varepsilon_d)$. If $\kappa_j$ occurs at time $t_j$ and involves $\cc_i$, then there is a collision $\kappa_j'$ in the timeline $\mathcal T_{\delta\circ\sigma}$ that occurs at time $-t_j$ and involves $\cc_i$ (though $\cc_i$ may have a different color in the coins diagrams of this timeline). In particular, $\kappa_d'$ is a right-wall collision, $\kappa_0'$ is a left-wall collision, and $\kappa_d'\lessdot_{\mathcal T_{\delta\circ\sigma}}\cdots\lessdot_{\mathcal T_{\delta\circ\sigma}}\kappa_0'$ is a saturated chain in ${\bf H}_{\mathcal T_{\delta\circ\sigma}}$. We have $\mathcal E(\kappa_d'\lessdot_{\mathcal T_{\delta\circ\sigma}}\cdots\lessdot_{\mathcal T_{\delta\circ\sigma}}\kappa_0')=(\varepsilon_d,\ldots,\varepsilon_1)$. Starting with this saturated chain, one can straightforwardly apply \Cref{lem:half-diamond,lem:diamond} to find that there is a transversal $\mathscr T'$ in ${\bf H}_{\mathcal T_{\delta\circ\sigma}}$ with energy composition $\mathcal E(\mathscr T')=(\varepsilon_d,\ldots,\varepsilon_1)$. Thus, $\widetilde\Omega(\delta\circ\sigma)$ is the orbit of $\Rot_{n,d}$ containing $(\varepsilon_d,\ldots,\varepsilon_1)$, which is $\rev(\widetilde\Omega(\sigma))$.
\end{proof}

\begin{lemma}\label{lem:scaling_factor}
For every $\mathcal O\in\Orb_{\Phi_{n,d}}$, we have $|\Omega(\mathcal O)|=\frac{d}{n}|\mathcal O|$. 
\end{lemma}

\begin{proof}
Fix $\mathcal O\in\Orb_{\Phi_{n,d}}$, and let $\mathcal T=(\sigma_t,t)_{t\in\ZZ}$ be a timeline such that $\sigma_0\in\mathcal O$. Consider a transversal $\mathscr T=(\kappa_0\lessdot_{\mathcal T}\kappa_1\lessdot_{\mathcal T}\cdots\lessdot_{\mathcal T}\kappa_d)$ of ${\bf H}_{\mathcal T}$, and let $\mathcal{E}(\mathscr T)=(\varepsilon_1,\ldots,\varepsilon_d)$. Then $\mathcal{E}(\mathscr T)$ is in the orbit $\Omega(\mathcal O)$. Let us define $\varepsilon_k$ for all $k\in\ZZ$ by declaring $\varepsilon_{i+d}=\varepsilon_i$. Let $t_j$ be the time when the collision $\kappa_j$ occurs. 

Consider the stones diagrams. Between times $t_0$ and $t_1$, the stone with the same color as $\cc_1$ slides along the cycle carrying $\vv_1$ until sliding from underneath $\vv_1$ to underneath $\vv_2$, which it carries until sliding underneath $\vv_3$, and so on until it finally slides underneath $\vv_{\varepsilon_1}$. The positions of $\vv_1,\ldots,\vv_{\varepsilon_1}$ throughout this interval of time are completely determined by the value of $\varepsilon_1$, the permutations $\stand(\mathcal T)$ and $\overline\stand(\mathcal T)$, and the residue of $t_0$ modulo $n$. It follows that $t_1-t_0$ is determined by $\varepsilon_1$, $\stand(\mathcal T)$, $\overline\stand(\mathcal T)$, and the residue of $t_0$ modulo $n$. Between times $t_1$ and $t_2$, the stone with the same color as $\cc_2$ slides along the cycle carrying $\vv_{\varepsilon_1+1}$ until sliding from underneath $\vv_{\varepsilon_1+1}$ to underneath $\vv_{\varepsilon_1+2}$, which it carries until sliding underneath $\vv_{\varepsilon_1+3}$, and so on until it finally slides underneath $\vv_{\varepsilon_1+\varepsilon_2}$. The positions of $\vv_{\varepsilon_1+1},\ldots,\vv_{\varepsilon_1+\varepsilon_2}$ throughout this interval of time are determined by the pair $(\varepsilon_1,\varepsilon_2)$, the permutations $\stand(\mathcal T)$ and $\overline\stand(\mathcal T)$, and the residue of $t_1$ modulo $n$. Thus, $t_2-t_1$ is determined by $(\varepsilon_1,\varepsilon_2)$, $\stand(\mathcal T)$, $\overline\stand(\mathcal T)$, and the residue of $t_0$ modulo $n$. In general, the values of $t_1-t_0,\ldots,t_d-t_{d-1}$ are determined by the energy composition $(\varepsilon_1,\ldots,\varepsilon_d)$, the permutations $\stand(\mathcal T)$ and $\overline\stand(\mathcal T)$, and the residue of $t_0$ modulo $n$. 

Let $p$ be the period of ${\bf H}_{\mathcal T}$. By \Cref{lem:ERot}, $p$ is equal to $|\Omega(\mathcal O)|$, the size of the orbit of $\Rot_{n,d}$ containing the composition $\mathcal{E}(\mathscr T)=(\varepsilon_1,\ldots,\varepsilon_d)$. Hence, $\varepsilon_1+\cdots+\varepsilon_p=\frac{p}{d}(\varepsilon_1+\cdots+\varepsilon_d)=pn/d$. Let $t_j^*$ be the time when the collision $\varphi^p(\kappa_j)$ occurs. Using \Cref{lem:half-diamond,lem:diamond}, we find that the edges in the half-diamond of ${\bf H}_{\mathcal T}$ between $\varphi^{i-1}(\kappa_0)$ and $\varphi^i(\kappa_0)$ both have energy $\varepsilon_i$. Therefore, by \Cref{lem:new}, the time between the collisions $\varphi^{i-1}(\kappa_0)$ and $\varphi^i(\kappa_0)$ is $(n-d)\varepsilon_i$. This shows that $\varphi^p(\kappa_0)$ occurs at time $t_0^*=t_0+(n-d)(\varepsilon_1+\cdots+\varepsilon_p)=t_0+pn(n-d)/d$. Because $p$ is the size of an orbit of $\Rot_{n,d}$, it is divisible by $d/\gcd(n,d)$; this implies that $t_0^*\equiv t_0\pmod{n}$. By the definition of $p$, the transversal $\varphi^p(\mathscr T)$ has the same energy composition $(\varepsilon_1,\ldots,\varepsilon_d)$ as $\mathscr T$. Since the permutations $\stand(\mathcal T)$ and $\overline\stand(\mathcal T)$ only depend on $\mathcal T$, it follows from the preceding paragraph that $t_j^*-t_{j-1}^*=t_j-t_{j-1}$ for all $1\leq j\leq d$; consequently, $t_j^*=t_j+pn(n-d)/d$ for all $0\leq j\leq d$. From this, we deduce that $\sigma_t=\sigma_{t+pn(n-d)/d}$ for all $t\in\ZZ$. In fact, $pn/d$ is the smallest positive integer $\ell$ such that $\sigma_t=\sigma_{t+\ell (n-d)}$ for all $t\in\ZZ$ (otherwise, we could reverse this argument to find that the period of ${\bf H}_{\mathcal T}$ is smaller than $p$).  

According to \eqref{eq:Phi_Order}, we have $\Phi_{n,d}^{pn/d}=\nu_{pn(n-d)/d}\cdots\nu_2\nu_1$, so $\Phi_{n,d}^{pn/d}(\sigma_0)=\sigma_{pn(n-d)/d}=\sigma_0$. Hence, $|\mathcal O|$ divides $pn/d$. On the other hand, since \Cref{lem:Phi_Divisible} tells us that $|\mathcal O|$ is divisible by $n/\gcd(n,d)$, we can use \eqref{eq:Phi_Order} to find that \[\sigma_{0}=\Phi_{n,d}^{|\mathcal O|}(\sigma_0)=(\nu_{|\mathcal O|(n-d)}\cdots\nu_2\nu_1)(\sigma_0)=\sigma_{|\mathcal O|(n-d)}.\] Since $|\mathcal O|(n-d)$ is divisible by $n$ (by \Cref{lem:Phi_Divisible}), we have $\nu_{t+|\mathcal O|(n-d)}=\nu_t$ for all integers $t$. Consequently, $\sigma_t=\sigma_{t+|\mathcal O|(n-d)}$ for all integers $t$. Appealing to the last sentence in the previous paragraph, we deduce that $|\mathcal O|\geq pn/d=\frac{n}{d}|\Omega(\mathcal O)|$. As $|\mathcal O|$ divides $pn/d$, the proof is complete.  
\end{proof}

Recall that \Cref{lem:depends_orbit} tells us that $\stand(\mathcal T_\sigma)$ and $\overline \stand(\mathcal T_\sigma)$ only depend on the orbit of $\Phi_{n,d}$ containing $\sigma$. In order to complete the proof of \Cref{prop:PhiRot}, we just need to show that $|\Omega^{-1}(\widehat{\mathcal O})|= d!(n-d)!$ for every $\widehat{\mathcal O}\in\Orb_{\Rot_{n,d}}$. We will do this by showing that for each pair of permutations $(\rho,\overline\rho)\in S_d\times S_{n-d}$, there exists a unique orbit $\mathcal O\in\Omega^{-1}(\widehat{\mathcal O})$ such that $\stand(\mathcal T_\sigma)=\rho$ and $\overline\stand(\mathcal T_\sigma)=\overline \rho$ for every $\sigma\in\mathcal O$. We start by proving existence; uniqueness will then follow from a simple counting argument. We implore the reader to consult \Cref{exam:5} while reading the proof of the next lemma. 

\begin{lemma}\label{lem:inequality}
Suppose $\widehat{\mathcal O}\in\Orb_{\Rot_{n,d}}$ and $(\rho,\overline\rho)\in S_d\times S_{n-d}$. There exists an orbit $\mathcal O\in\Omega^{-1}(\widehat{\mathcal O})$ such that $\stand(\mathcal T_\sigma)=\rho$ and $\overline\stand(\mathcal T_\sigma)=\overline \rho$ for every $\sigma\in\mathcal O$. 
\end{lemma}

\begin{proof}
If $d=1$, then the result is obvious because $\widehat{\mathcal O}=\Comp_{1}(n)=\{(n)\}$. Therefore, we may assume $d\geq 2$ and proceed by induction on $d$. It follows from \Cref{lem:reverse} that the following two statements are equivalent: 
\begin{enumerate}
\item There exists an orbit $\mathcal O\in\Omega^{-1}(\widehat{\mathcal O})$ such that $\stand(\mathcal T_\sigma)=\rho$ and $\overline\stand(\mathcal T_\sigma)=\overline \rho$ for every $\sigma\in\mathcal O$.
\item There exists an orbit $\mathcal O\in\Omega^{-1}(\rev(\widehat{\mathcal O}))$ such that $\stand(\mathcal T_\sigma)=\rev(\rho)$ and $\overline\stand(\mathcal T_\sigma)=\rev(\overline \rho)$ for every $\sigma\in\mathcal O$.
\end{enumerate} 
Therefore, we may assume\footnote{This assumption might seem innocuous, but it is actually imperative for our argument. Thus, \cref{lem:reverse} really is quite crucial.} without loss of generality that the number $1$ appears to the left of the number $2$ in the one-line notation of $\rho$ (otherwise, replace $\rho$, $\overline{\rho}$, and $\widehat{\mathcal O}$ by $\rev(\rho)$, $\rev(\overline\rho)$, and $\rev(\widehat{\mathcal O})$, respectively).   

Since $d<n$, every composition in $\widehat{\mathcal O}$ has a part that is strictly greater than $1$. Thus, we may choose a composition $(\varepsilon_1,\ldots,\varepsilon_d)\in\widehat{\mathcal O}$ such that $\varepsilon_2\geq 2$. We will also assume for simplicity that $\varepsilon_1\geq 2$; the case when $\varepsilon_1=1$ is similar. Let $\rho'$ be the permutation in $S_{d-1}$ obtained from $\rho$ by deleting the entry $1$ and decreasing the remaining entries by $1$. Let $\overline \rho'$ be the permutation in $S_{d-\varepsilon_1+1}$ obtained from $\overline \rho$ by deleting the entries $1,\ldots,\varepsilon_1-1$ and decreasing the remaining entries by $\varepsilon_1-1$. Let $\widehat{\mathcal O}'$ be the orbit of $\Rot_{n-\varepsilon_1,d-1}$ containing $(\varepsilon_2,\ldots,\varepsilon_d)$. By induction, there exists an orbit $\mathcal O'\in\Omega_{n-\varepsilon_1,d-1}^{-1}(\widehat{\mathcal O}')$ such that $\stand(\mathcal T_{\sigma'})=\rho'$ and $\overline\stand(\mathcal T_{\sigma'})=\overline \rho'$ for every $\sigma'\in\mathcal O'$ (the timeline $\mathcal T_{\sigma'}$ is defined with the parameters $n-\varepsilon_1$ and $d-1$ replacing $n$ and $d$). 

Fix $\sigma_0'\in\mathcal O'$, and consider the timeline $\mathcal T_{\sigma_0'}=(\sigma_t',t)_{t\in\ZZ}$ (defined with the parameters $n-\varepsilon_1$ and $d-1$). Let us identify $\Path_{n-\varepsilon_1}$ with the subgraph of $\Path_n$ obtained by deleting the vertices $v_1,\ldots,v_{\varepsilon_1}$. Thus, the leftmost vertex in $\Path_{n-\varepsilon_1}$ is $v_{\varepsilon_1+1}$, and the replicas appearing in the stones diagrams of states in $\mathcal T_{\sigma_0'}$ are $\vv_{\varepsilon_1+1},\ldots,\vv_n$. Let 
$\kappa_1'\lessdot_{\mathcal T_{\sigma_0'}}\cdots\lessdot_{\mathcal T_{\sigma_0'}}\kappa_d'$ be a transversal of ${\bf H}_{\mathcal T_{\sigma_0'}}$ with energy composition $(\varepsilon_2,\ldots,\varepsilon_d)$. Let $k+1$ be the first time after the collision $\kappa_1'$ when the leftmost coin moves. Because $\varepsilon_2\geq 2$, the leftmost coin occupies $v_{\varepsilon_1+1}$ in the coins diagram of $(\sigma_k',k)$ and occupies $v_{\varepsilon_1+2}$ in the coins diagram of $(\sigma_{k+1}',k+1)$. Let $v_\eta$ be the vertex of $\Path_{n-\varepsilon_1}$ such that $\sigma_k'(v_\eta)=d+k\in\ZZ/(n-\varepsilon_1)\ZZ$. In the stones diagram of $(\sigma_k',k)$, the replica $\vv_\eta$ sits one space clockwise from the consecutive block of stones. 

We can construct the stones diagram of a state $(\mu,m)$ with $\mu\in\Lambda_{\Path_n}$ from the stones diagram of $(\sigma_k',k)$ by inserting vertices to replace $\Cycle_{n-\varepsilon_1}$ by $\Cycle_n$, adding one stone, and adding the replicas $\vv_1,\ldots,\vv_{\varepsilon_1}$. When we do this, we make sure to keep the stones on the consecutive block of vertices $m+d,\ldots,m+1$, and we make sure that the replicas that were sitting on stones remain on stones. We place the replica $\vv_{\varepsilon_1}$ on the newly inserted stone. We can also ensure that $\stand(\mu,m)=\rho$ and $\overline\stand(\mu,m)=\overline\rho$, and we can choose $m$ so that $\vv_{\eta}$ is on the vertex $m+d+1$. (In fact, these conditions uniquely determine $\mu$ and uniquely determine $m$ modulo $n$.) Let $\sigma_0$ be the unique labeling in $\Lambda_{\Path_n}$ such that the timeline $\mathcal T_{\sigma_0}$ contains the state $(\mu,m)$. 

Consider watching the coins diagrams of the states in $\mathcal T_{\sigma_0}$ evolve over time. At time $m$, the coins $\cc_1$ and $\cc_2$ occupy $v_{\varepsilon_1}$ and $v_{\varepsilon_1+1}$, respectively. At time $m+1$, the coin $\cc_2$ moves to $v_{\varepsilon_1+2}$, and $\cc_1$ stays on $v_{\varepsilon_1}$ (we are using the fact that $1$ appears to the left of $2$ in $\rho$). This implies (since $\varepsilon_1\geq 2$) that the last collision involving $\cc_1$ that occurred at or before time $m$ must have been a two-coins collision involving $\cc_1$ and $\cc_2$; let us call this collision $\kappa_1$ and say that it occurred at time $m'$. After time $m+1$, $\cc_1$ will move leftward until reaching $v_1$, where it will take part in a left-wall collision $\kappa_0^*$. Meanwhile, $\cc_2$ will travel rightward until reaching $v_{\varepsilon_1+\varepsilon_2}$, where it will collide with $\cc_3$ in a two-coins collision $\kappa_2$. Then $\cc_3$ will move rightward until reaching $v_{\varepsilon_1+\varepsilon_2+\varepsilon_3}$, where it will collide with $\cc_4$ in a two-coins collision $\kappa_3$, and so on. Eventually, $\cc_d$ moves rightward and takes part in a right-wall collision $\kappa_d$. The key observation here is that throughout this process, in the stones diagrams, any stone carrying a replica $\vv_\ell$ with $\ell\geq \varepsilon_1+2$ will just slide through any replica $\vv_{\ell'}$ with $\ell'\leq\varepsilon_1$. This means that the replicas $\vv_1,\ldots,\vv_{\varepsilon_1}$ that we inserted when passing from $\Cycle_{n-\varepsilon_1}$ to $\Cycle_n$ will not affect where the collisions $\kappa_2,\kappa_3,\ldots,\kappa_d$ occur. This is why $\kappa_2,\ldots,\kappa_d$ occur at the same places (though possibly at different times) as $\kappa_2',\ldots,\kappa_d'$, respectively. We find that $\mathcal E(\kappa_1\lessdot_{\mathcal T_{\sigma_0}}\kappa_0^*)=\varepsilon_1$ and $\mathcal E(\kappa_1\lessdot_{\mathcal T_{\sigma_0}}\kappa_2\lessdot_{\mathcal T_{\sigma_0}}\cdots\lessdot_{\mathcal T_{\sigma_0}}\kappa_d)=(\varepsilon_2,\ldots,\varepsilon_d)$. The edge $\kappa_1\lessdot_{\mathcal T_{\sigma_0}}\kappa_0^*$ is the top edge in a half-diamond; let $\kappa_0\lessdot_{\mathcal T_{\sigma_0}}\kappa_1$ be the bottom edge of the same half-diamond. Then \Cref{lem:half-diamond} tells us that $\mathcal E(\kappa_0\lessdot_{\mathcal T_{\sigma_0}}\kappa_1)=\varepsilon_1$, so the transversal $\kappa_0\lessdot_{\mathcal T_{\sigma_0}}\kappa_1\lessdot_{\mathcal T_{\sigma_0}}\cdots\lessdot_{\mathcal T_{\sigma_0}}\kappa_d$ of ${\bf H}_{\mathcal T_{\sigma_0}}$ has energy composition $(\varepsilon_1,\ldots,\varepsilon_d)$. Thus, $\widetilde\Omega(\sigma_0)$ is the orbit of $\Rot_{n,d}$ containing $(\varepsilon_1,\ldots,\varepsilon_d)$. If $\mathcal O$ is the orbit of $\Phi_{n,d}$ containing $\sigma_0$, then $\Omega(\mathcal O)=\widehat{\mathcal O}$. Furthermore, $\stand(\mathcal T_{\sigma_0})=\rho$ and $\overline\stand(\mathcal T_{\sigma_0})=\overline\rho$. According to \Cref{lem:depends_orbit}, we have $\stand(\mathcal T_{\sigma})=\rho$ and $\overline\stand(\mathcal T_{\sigma})=\overline\rho$ for all $\sigma\in\mathcal O$. 
\end{proof}

\begin{figure}[ht]
  \begin{center}{\includegraphics[height=8cm]{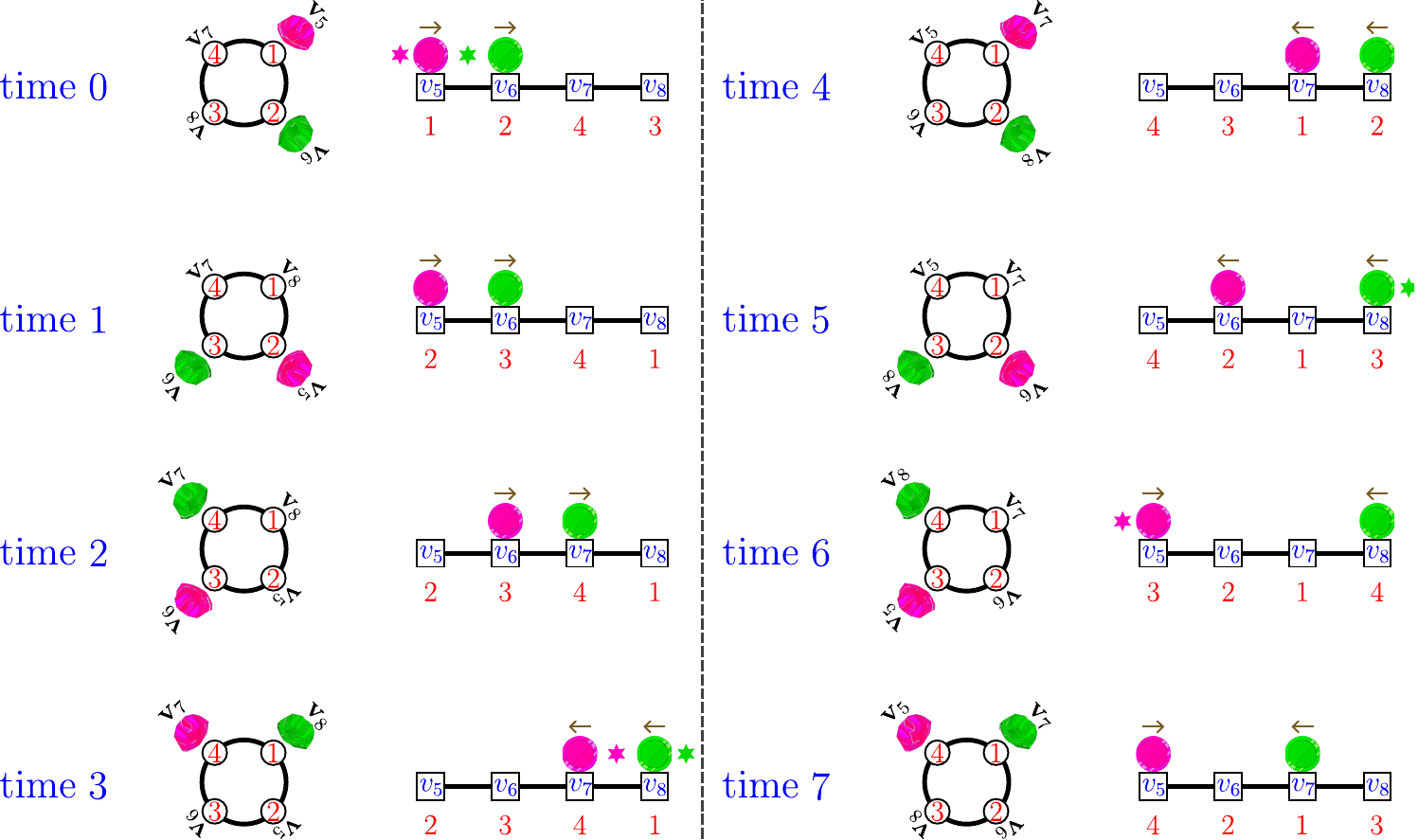}}
  \end{center}
  \caption{The stones diagrams and coins diagrams of the states at times $0,1,\ldots,7$ in the timeline $\mathcal T_{\sigma_0'}$ from \Cref{exam:5}. }\label{Fig8}
\end{figure}

\begin{example}\label{exam:5}
Let us illustrate the proof of \Cref{lem:inequality}. Suppose $n=8$, $d=3$, $\rho=132$, and $\overline \rho=52413$. Let $\widehat{\mathcal O}$ be the orbit of $\Rot_{8,3}$ containing the composition $(\varepsilon_1,\varepsilon_2,\varepsilon_3)=(4,3,1)$. Note that $1$ appears before $2$ in $\rho$ and that $\varepsilon_2\geq 2$. We have $\rho'=21$ and $\overline\rho'=21$. 

We can choose $\sigma_0'$ to be the labeling such that the stones diagrams and coins diagrams of the states of $\mathcal T_{\sigma_0'}$ at times $0,1,\ldots,7$ are shown in \Cref{Fig8}. One can check that the states in this timeline are periodic with period $8$. We can choose the transversal $\kappa_1'\lessdot_{\mathcal T_{\sigma_0'}}\kappa_2'\lessdot_{\mathcal T_{\sigma_0'}}\kappa_3'$ so that $\kappa_1'$ is the left-wall collision at time $0$, $\kappa_2'$ is the two-coins collision at time $3$, and $\kappa_3'$ is the right-wall collision at time $5$.  We have $k=1$ and $\eta=7$. 

\Cref{Fig9} illustrates how we construct the stones diagram of $(\mu,m)$ from that of $(\sigma_1',1)$. In this example, $m=2$. Four vertices were inserted to transform $\Cycle_4$ into $\Cycle_8$, and the vertices were then renamed. Since $\eta=7$, we have placed $\vv_7$ on the vertex $m+d+1=6$. Note that the standardization of $\vv_4,\vv_6,\vv_5$ is $132=\rho$ and that the standardization of $\vv_8,\vv_2,\vv_7,\vv_1,\vv_3$ is $52413=\overline\rho$. 

\Cref{Fig10} shows the stones diagrams and coins diagrams of the states in $\mathcal T_{\sigma_0}$ at times $0,\ldots,11$. (The labelings of the states in this timeline are actually periodic with period $40$, but we chose not to draw the diagrams of $40$ states.) The collision $\kappa_1$ involves $\cc_1$ and $\cc_2$ and occurs at time $0$. Then $\kappa_0^*$ is the left-wall collision at time $9$. The collision $\kappa_2$ involves $\cc_2$ and $\cc_3$ and occurs at time $6$, while $\kappa_3$ is the right-wall collision at time $11$. Observe that $\mathcal E(\kappa_1\lessdot_{\mathcal T_{\sigma_0}}\kappa_0^*)=4=\varepsilon_1$ and $\mathcal E(\kappa_1\lessdot_{\mathcal T_{\sigma_0}}\kappa_2\lessdot_{\mathcal T_{\sigma_0}}\kappa_3)=(3,1)=(\varepsilon_2,\varepsilon_3)$.  
\hfill $\lozenge$
\end{example}

\begin{figure}[ht]
  \begin{center}{\includegraphics[height=3.5cm]{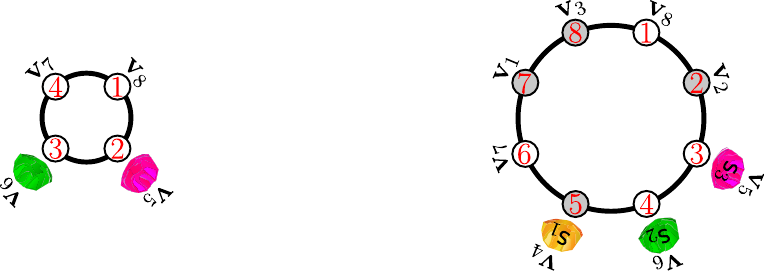}}
  \end{center}
  \caption{On the left is the stones diagram of $(\sigma_1',1)$ from \Cref{exam:5}. On the right is the stones diagram of $(\mu,2)$, which is constructed from that of $(\sigma_1',1)$ by inserting fours new vertices (shaded), a new stone (gold), and the new replicas $\vv_1,\vv_2,\vv_3,\vv_4$. }\label{Fig9}
\end{figure}

\begin{figure}[ht]
  \begin{center}{\includegraphics[width=\linewidth]{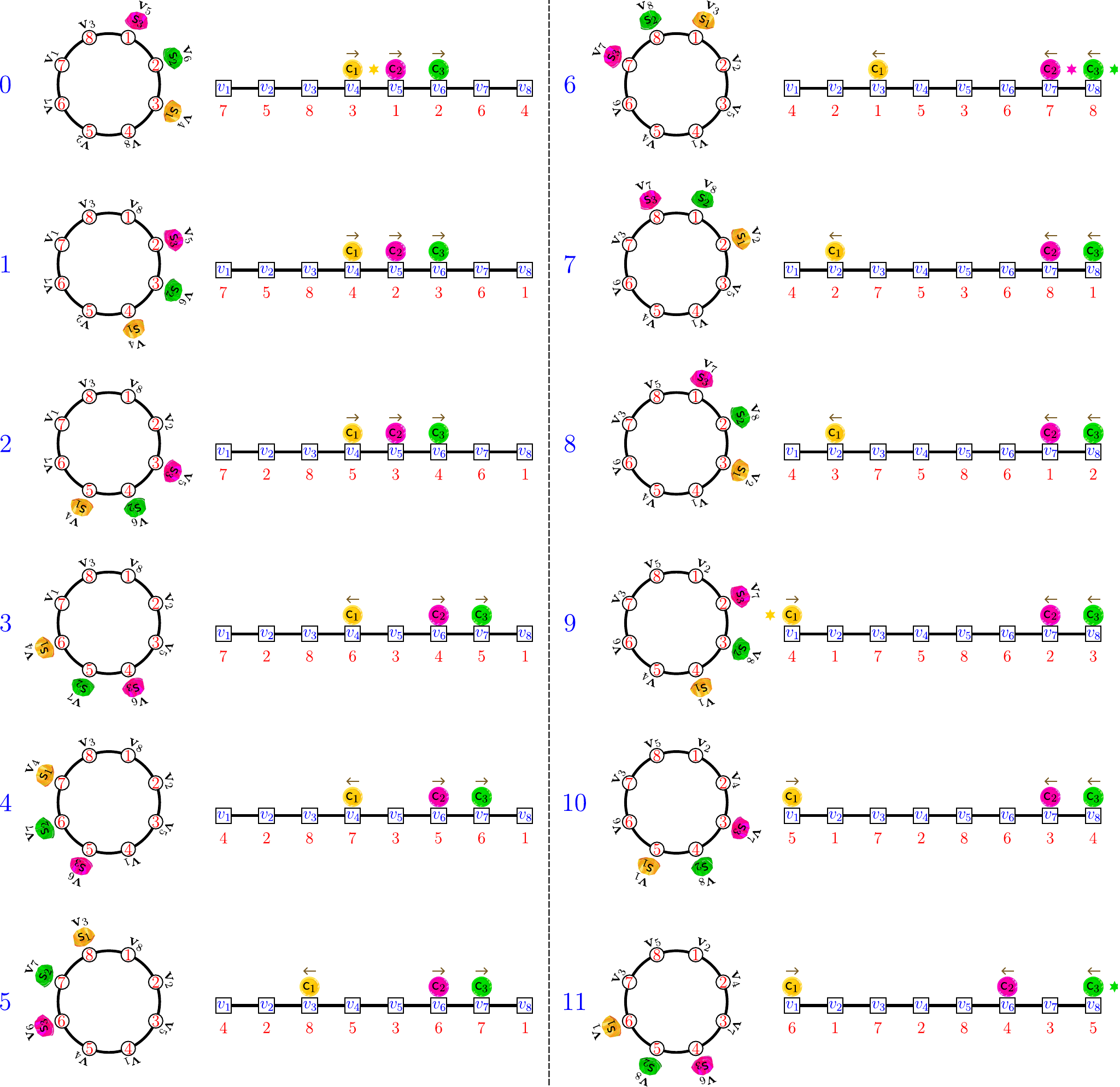}}
  \end{center}
  \caption{The stones diagrams and coins diagrams of the states at times $0,1,\ldots,11$ in the timeline $\mathcal T_{\sigma_0}$ from \Cref{exam:5}. Blue numbers indicate times. Collisions are represented by color-coded stars. }\label{Fig10}
\end{figure}

\begin{proof}[Proof of \Cref{prop:PhiRot}]
We know by \Cref{lem:scaling_factor} that $|\Omega(\mathcal O)|=\frac{d}{n}|\mathcal O|$ for every $\mathcal O\in\Orb_{\Phi_{n,d}}$. It follows from \Cref{lem:inequality} that $|\Omega^{-1}(\widehat{\mathcal O})|\geq d!(n-d)!$ for every $\widehat{\mathcal O}\in\Orb_{\Rot_{n,d}}$. Therefore, 
\[n!=|\Lambda_{\Path_n}|=\sum_{\mathcal O\in\Orb_{\Phi_{n,d}}}|\mathcal O|=\sum_{\widehat{\mathcal O}\in\Orb_{\Rot_{n,d}}}\sum_{\mathcal O\in\Omega^{-1}(\widehat{\mathcal O})}|\mathcal O|=\sum_{\widehat{\mathcal O}\in\Orb_{\Rot_{n,d}}}|\Omega^{-1}(\widehat{\mathcal O})|\cdot\frac{n}{d}|\widehat{\mathcal O}|\] \[\geq d!(n-d)!\frac{n}{d}\sum_{\widehat{\mathcal O}\in\Orb_{\Rot_{n,d}}}|\widehat{\mathcal O}|=n(d-1)!(n-d)!|\Comp_d(n)|=n(d-1)!(n-d)!\binom{n-1}{d-1}=n!.\] This inequality must actually be an equality, so we must have $|\Omega^{-1}(\widehat{\mathcal O})|=d!(n-d)!$ for every $\widehat{\mathcal O}\in\Orb_{\Rot_{n,d}}$. 
\end{proof}

As discussed at the end of \Cref{subsec:reformulation}, \Cref{prop:PhiRot} implies \Cref{thm:main}.

\section{Orbit Structure of Broken Promotion}\label{sec:orbit_broken}

In this final section, we prove \Cref{thm:broken_main2,thm:broken_main}, which describe the orbit structure of $\cyc\Bro_B$ for particular choices of the subset $B\subseteq\ZZ/n\ZZ$. 

\begin{proof}[Proof of \Cref{thm:broken_main2}]
Let $\beta$ be the acyclic orientation of $\Cycle_n$ whose unique source is $d$ and whose unique sink is $n$. To ease notation, let $F(q)=n(d-1)!(n-d-1)![n-d]_{q^d}{n-1\brack d-1}_q$. \Cref{thm:main} tells us that $\TPro_\beta$ has order $d(n-d)$ and that the triple $(\Lambda_{\Path_n},\TPro_\beta,F(q))$ exhibits the cyclic sieving phenomenon. Since the sizes of the orbits of $\TPro_\beta$ are all divisible by $d$ (by \Cref{prop:divisibility}), it follows that $\TPro_\beta^{d}$ has order $n-d$ and that the triple $(\Lambda_{\Path_n},\TPro_\beta^{d},F(q))$ also exhibits the cyclic sieving phenomenon. 
By \Cref{rem:Bro_d}, we have \[\TPro_\beta^d=\left(\cyc^{-1}\Bro_{\{1,\ldots,d\}}^{-1}\right)^n=(\Bro_{\{1,\ldots,d\}}\cyc)^{-n}=\cyc^{-1}(\cyc\Bro_{\{1,\ldots,d\}})^{-n}\cyc,\] so $\TPro_\beta^d$ and $(\cyc\Bro_{\{1,\ldots,d\}})^{n}$ have the same orbit structure. Consequently, $(\cyc\Bro_{\{1,\ldots,d\}})^{n}$ has order $n-d$, and the triple $(\Lambda_{\Path_n},(\cyc\Bro_{\{1,\ldots,d\}})^{n},F(q))$ satisfies the cyclic sieving phenomenon. It follows immediately from \Cref{prop:homomesy} that the sizes of the orbits of $\cyc\Bro_{\{1,\ldots,d\}}$ are all divisible by $n$. Therefore, $\cyc\Bro_{\{1,\ldots,d\}}$ has order $(n-d)n$, and if $\{k_i^{m_i}:1\leq i\leq \ell\}$ is the multiset of orbit sizes of $(\cyc\Bro_{\{1,\ldots,d\}})^n$, then $\{(nk_i)^{m_i/n}:1\leq i\leq \ell\}$ is the multiset of orbit sizes of $\cyc\Bro_{\{1,\ldots,d\}}$. According to \Cref{lem:CSP_technical}, the triple $(\Lambda_{\Path_n},\cyc\Bro_{\{1,\ldots,d\}},\frac{1}{n}[n]_{q^{n-d}}F(q))$ exhibits the cyclic sieving phenomenon. This completes the proof.  
\end{proof}

\begin{proof}[Proof of \Cref{thm:broken_main}]
Let $d$, $n$, $s_i$, and $\mathscr R$ be as in the statement of the theorem. Let $\beta$ be the acyclic orientation of $\Cycle_n$ whose sources are the elements of the set $\mathscr S=\{s_1,\ldots,s_d\}$ and whose sinks are the elements of $\mathscr S-1$. Let $F(q)=n(d-1)!(n-d-1)![n-d]_{q^d}{n-1\brack d-1}_q$. \Cref{thm:main} tells us that $\TPro_\beta$ has order $d(n-d)$ and that the triple $(\Lambda_{\Path_n},\TPro_\beta,F(q))$ exhibits the cyclic sieving phenomenon. Since the sizes of the orbits of $\TPro_\beta$ are all divisible by $n-d$ (by \Cref{prop:divisibility}), it follows that $\TPro_\beta^{n-d}$ has order $d$ and that the triple $(\Lambda_{\Path_n},\TPro_\beta^{n-d},F(q))$ also exhibits the cyclic sieving phenomenon. If we set $\gamma=n-d$, $q=n$, and $r=0$ in \Cref{prop:TProcycBro_R}, we find that $\TPro_\beta^{n-d}=(\cyc\Bro_{\mathscr R})^n$. It follows from \Cref{prop:homomesy} that the sizes of the orbits of $\cyc\Bro_{\mathscr R}$ are all divisible by $n$. Therefore, $\cyc\Bro_{\mathscr R}$ has order $dn$, and if $\{k_i^{m_i}:1\leq i\leq \ell\}$ is the multiset of orbit sizes of $(\cyc\Bro_{\mathscr R})^n$, then $\{(nk_i)^{m_i/n}:1\leq i\leq \ell\}$ is the multiset of orbit sizes of $\cyc\Bro_{\mathscr R}$. According to \Cref{lem:CSP_technical}, the triple $(\Lambda_{\Path_n},\cyc\Bro_{\mathscr R},\frac{1}{n}[n]_{q^d}F(q))$ exhibits the cyclic sieving phenomenon.  
\end{proof}

\section{Future Directions} 

\Cref{thm:toric_main} determines the orbit structure of toric promotion when $G$ is a forest. It is still open to understand the dynamics of toric promotion for other graphs, including cycle graphs. 

\Cref{thm:main} determines the orbit structure of any permutoric promotion operator when $G$ is a path. It would be interesting to gain a better understanding of $\TPro_\pi$ when $G$ is another type of tree, even in the special case when $\pi^{-1}$ has $2$ cyclic descents. A natural place to start could be the case when $G$ is obtained from $\Path_{n-1}$ by adding a new vertex that is adjacent to $v_{n-2}$ (i.e., $G$ is the Dynkin diagram of type $D_n$).

\section*{Acknowledgements}
Colin Defant was supported by the National Science Foundation under Award No.\ 2201907 and by a Benjamin Peirce Fellowship at Harvard University. He thanks Tom Roby for initially suggesting the generalization from toric promotion to permutoric promotion. Hugh Thomas was supported by NSERC Discovery Grant RGPIN-2022-03960 and the Canada Research Chairs program, grant number CRC-2021-00120. We thank Nathan Williams for inspiring conversations.

\end{document}